\documentclass[10pt, notitlepage]{article}
\usepackage{amsmath, amssymb, pgfplots,yfonts, setspace, enumitem,titlesec}
\usepackage[toc,page]{appendix}
\usepackage{hyperref}
\usepackage{pst-node}
\usepackage{tikz-cd} 
\usepackage[margin=1.5in]{geometry}
\usepackage{cite}

\def\tr{\mathop{\hbox{tr}}}

\def\id{\mathop{\hbox{id}}}

\def\im{\mathop{\hbox{im}}}

\def\Span{\mathop{\hbox{span}}}

\def\Pol{\mathop{\hbox{Pol}}}

\def\fB{\mathcal{B}}
\def\fH{\mathcal{H}}

\def\fK{\mathcal{K}}

\def\C{\mathbb{C}}

\def\H{\mathbb{H}}

\def\G{\mathbb{G}}

\usepackage{mathtools}

\makeatletter
\DeclareRobustCommand\widecheck[1]{{\mathpalette\@widecheck{#1}}}
\def\@widecheck#1#2{%
    \setbox\z@\hbox{\m@th$#1#2$}%
    \setbox\tw@\hbox{\m@th$#1%
       \widehat{%
          \vrule\@width\z@\@height\ht\z@
          \vrule\@height\z@\@width\wd\z@}$}%
    \dp\tw@-\ht\z@
    \@tempdima\ht\z@ \advance\@tempdima2\ht\tw@ \divide\@tempdima\thr@@
    \setbox\tw@\hbox{%
       \raise\@tempdima\hbox{\scalebox{1}[-1]{\lower\@tempdima\box
\tw@}}}%
    {\ooalign{\box\tw@ \cr \box\z@}}}
\makeatother
\usepackage[utf8]{inputenc}
\usepackage[english]{babel}
\usepackage{amsthm}

\theoremstyle{plain}
\newtheorem{thm}{Theorem}[section]
\newtheorem{lem}[thm]{Lemma}
\newtheorem{prop}[thm]{Proposition}
\newtheorem{cor}[thm]{Corollary}
\newtheorem{ques}[thm]{Question}

\theoremstyle{definition}
\newtheorem{defn}[thm]{Definition}
\newtheorem{rem}[thm]{Remark}
\newtheorem{eg}[thm]{Example}
\makeatletter
\newcommand*{\toccontents}{\@starttoc{toc}}
\makeatother
\titleformat{\chapter}
  {\normalfont\bfseries}{\thechapter}{1em}{}
\titlespacing*{\chapter}{0pt}{3.5ex plus 1ex minus .2ex}{2.3ex plus .2ex}
\pgfplotsset{compat=1.17}

\begin{document}

\markboth{Benjamin Anderson-Sackaney}{Ideals of $L^1$-algebras}


\title{On Ideals of $L^1$-algebras of Compact Quantum Groups}

\author{Benjamin Anderson-Sackaney\\ {\small Email: b8anders@uwaterloo.ca}\\
{\small Department of Pure Mathemmatics, University of Waterloo,}\\
{\small 200 University Avenue, Waterloo, N2L 3G1, Canada}}



\maketitle

\begin{abstract}
    We develop a notion of a non-commutative hull for a left ideal of the $L^1$-algebra of a compact quantum group $\G$. A notion of non-commutative spectral synthesis for compact quantum groups is proposed as well. It is shown that a certain Ditkin's property at infinity (which includes those $\G$ where the dual quantum group $\widehat{\G}$ has the approximation property) is equivalent to every hull having synthesis. We use this work to extend recent work of White that characterizes the weak$^*$ closed ideals of a measure algebra of a compact group to those of the measure algebra of a coamenable compact quantum group. In the sequel, we use this work to study bounded right approximate identities of certain left ideals of $L^1(\G)$ in relation to coamenability of $\G$.
\end{abstract}

\section{Introduction}
Describing the ideals of a Banach algebra is a fundamental problem. As is done with Hilbert's Nullstellensatz, the closed ideals of a semi-simple Tauberian commutative Banach algebra $A$ of a certain type (without going into details) can be distinguished by their zero sets in the Gelfand spectrum $\sigma(A)$. A semi-simple commutative Banach algebra $A$ is Tauberian if $\{a : \hat{a} ~ \text{has compact support}\}$, where $\hat{a}$ is the Gelfand transform of $a$, is dense in $A$. More precisely, every closed ideal $I$ of such $A$ has that $I = I(E) = \{a\in A : \hat{a}|_E = 0\}$ for some closed subset $E\subseteq\sigma(A)$, and $E$ is called the {\bf hull} of $I$.

This correspondence lends itself nicely to commutative Banach algebras studied in abstract harmonic analysis. The Fourier algebra $A(G)\subseteq C_0(G)$ of a locally compact group $G$ is the commutative Banach algebra of coefficient functions of the left regular representation of $G$, and naturally identifies with the predual of the group von Neumann algebra $VN(G)$ (for the basics of Fourier algebras, see \cite{E64, KL19}). Alternatively, we have that $A(G) = L^1(\widehat{G})$ where $\widehat{G}$ is the quantum group dual of $G$ (see Section $2$) where we note that if $G$ is abelian then $\widehat{G}$ is the Pontryagin dual of $G$ and $L^1(\widehat{G}) = A(G)$ has a bounded approximate identity. In general, $A(G)$ has a bounded approximate identity if and only if $G$ is amenable \cite{L68}, $\sigma(A(G)) = G$, and $A(G)$ is Tauberian. It turns out that if $G = \Gamma$ is discrete and amenable, then every closed ideal $I\subseteq A(\Gamma)$ has $I = I(E)$ where $E = \mathrm{hull}(I)\subseteq \Gamma$.

More generally, for a closed subset $E\subseteq G$, we will write
$$I(E) = \{u\in A(G) : u|_E = 0\}$$
and
$$j(E) = \{ u\in A(G) : u ~ \text{has compact support disjoint from $E$}\}.$$
The ideal $I(E)$ is always closed. Since $A(G)$ is Tauberian, for any ideal $I\subseteq A(G)$ we have $j(E)\subseteq I\subseteq I(E)$ where $E = \mathrm{hull}(I)$ \cite[Chap. X Section 1]{HR63}. The closed subset $E$ is said to be a {\bf set of synthesis} if $\overline{j(E)} = I(E)$ and so with this language, the closed ideal structure of $A(G)$ is completely characterized when every closed subset of $G$ is a set of synthesis. The locally compact groups where such a thing holds have been characterized.
\begin{thm}\cite{Lau01}\label{IdealsA(G)}
    Let $G$ be a locally compact group. Then every closed subset of $G$ is a set of synthesis if and only if $G$ is discrete and $u\in \overline{uA(G)}$ for all $u\in A(G)$.
\end{thm}
Whenever $u\in \overline{uA(G)}$ for all $u\in A(G)$, we say $G$ has {\bf Ditkin's property at infinity} or {\bf property $D_\infty$}. This property holds for a broad range of groups, which clearly includes all of those which admit an approximate identity, but is poorly understood. Indeed, there are no known examples of locally compact groups without property $D_\infty$ (see Section $3.3$ for more on property $D_\infty$ and its quantization). On the other hand, we understand the closed ideals of $A(G)$ for many examples of discrete groups in the literature (which includes all discrete groups with the approximation property (see Section $2.4$).

Because of Schur's lemma, the Gelfand spectrum of a commutative Banach algebra $A$ is the set of irreducible representations $A$. So, it seems natural to try to use irreducible representations to try to build a ``quantum hull'' for a left ideal of a non-commutative Banach algebra to glean information on its structure. Such a thing was achieved for group algebras of compact groups. Let us fix a compact group $G$. Recall that for a unitary representation $\pi : G\to\fB(\fH_\pi)$, the corresponding $L^1$-representation is given by $\pi(f) = \int_G f(t)\pi(t)\,dt$ for $f\in L^1(G)$. The closed left ideals have a representation theoretic description as follows.
\begin{thm}\cite{HR63}\label{strClosedIdealsCpct}
    The closed left ideals of $L^1(G)$ are of the form
    $$I(E) = \{ f\in L^1(G) : \pi(f)(E_\pi) = 0, ~ \pi\in Irr(G)\}$$
    where $E = (E_\pi)_{\pi\in Irr(G)}$ for subspaces $E_\pi\subseteq \fH_\pi$.
\end{thm}
The symbol $Irr(G)$ denotes the irreducible $*$-representations on $L^1(G)$.

In the general scheme of locally compact quantum groups (LCQGs), the compact and discrete quantum groups are dual to one another (cf. \cite{Wor} and \cite{R08}). So, it is reasonable to attempt to unify Theorems \ref{strClosedIdealsCpct} and \ref{IdealsA(G)} at the level of compact quantum groups (CQGs). Using the analogies found between the representation theory of compact quantum groups (CQGs) in general and compact groups (cf. \cite{Wor}), we formulate notions of hull and synthesis. In particular, we have that the sequences $E = (E_\pi)_{\pi\in Irr(\G)}$, where each $E_\pi$ is a subspace of the Hilbert space $\fH_\pi$ where $L^1(\G)$ acts by $\pi$ and $Irr(\G)$ are the irreducible $*$-representations on $L^1(\G)$, are hulls of left ideals of $L^1(\G)$ and we say $E$ is a set of synthesis if $\overline{j(E)} = I(E)$ where $j(E)$ and $I(E)$ are defined in Section $3.1$.  More precisely, let $\G$ be a CQG and $I\subseteq L^1(\G)$ a left ideal. We prove that there exists a sequence $E = (E_\pi)_{\pi\in Irr(\G)}$ such that:
\begin{itemize}
    \item $j(E)\subseteq I\subseteq I(E)$ (Proposition \ref{idealsinL^1});
    \item every $E = (E_\pi)_{\pi\in Irr(\G)}$ is a set of synthesis if and only if $f\in \overline{L^1(\G)*f}$ for every $f\in L^1(\G)$ (Theorem \ref{StrcClosedLeftIdeals}).
\end{itemize}
In particular, for every CQG $\G$ where $\widehat{\G}$ satisfies property (left) $D_\infty$, which is the property where $f\in \overline{L^1(\G)*f}$ for every $f\in L^1(\G)$, we establish a complete description of the closed left ideals of $L^1(\G)$. This includes many examples of CQGs from the literature, including $SU_q(2)$, free unitary quantum groups, free orthogonal quantum groups, etc (see Section $3.3$).

A coamenable CQG is a CQG $\G$ where $L^1(\G)$ admits a bounded approximate identity, and are CQGs such that their duals have property left $D_\infty$. In this case, our above structural result for closed left ideals of $L^1(\G)$ can be used to characterize the weak$^*$ closed left ideals of their quantum measure algebras. The measure algebra $M(G)$ of a locally compact group, is a unital Banach algebra that identifies naturally with the dual space of the continuous functions vanishing at infinity $C_0(G)$ via integration. Then, here, we have that $\overline{L^1(G)}^{wk*} = M(G)$ with respect to the $\sigma(M(G), C_0(G))$ topology. Switching perspectives to the duals of locally compact amenable $G$, the Fourier-Stieltjes algebra $B(G)$, which is the subalgebra of the bounded continuous functions $C_b(G)$ consisting of coefficient functions of continuous unitary representations, is a unital commutative Banach algebra that identifies with the dual space of the group $C^*$-algebra $C^*(G)$. Here, we have the analogous fact that $\overline{A(G)}^{wk*} = B(G)$ with respect to the $\sigma(B(G),C^*(G))$ topology. For coamenable CQGs, the quantum measure algebra $M(\G)$ is the dual space of the $C^*$-algebra $C(\G)$ of $\G$, and likewise satisfies $\overline{L^1(\G)}^{wk*} = M(\G)$. It is a generalization of both the measure algebra of a compact group and the Fourier-Stieltjes algebra of an amenable discrete group.

Obtaining information about the structure of a quantum measure algebra is generally very difficult. For instance, the Gelfand spectrum $\sigma(B(G))$ is often much larger than $G$, as exhibited by the Wiener-Pitt phenomenon for non-compact abelian groups (see \cite{W75, W72}). We can learn about some of the structure of a quantum measure algebra, however, if the natural weak$^*$ topology is taken into account. Recently, White \cite{Wh18} described the weak$^*$ closed left ideals of the measure algebra $M(G)$ of a compact group $G$ by exploiting weak$^*$ approximation by elements in $L^1(G)$. In order to get a description similar
to the one in Theorem \ref{wk*IdealsMeasures}, we show Whites techniques extend directly to coamenable CQGs. Indeed, we prove the following.
\begin{itemize}
    \item Let $\G$ be a coamenable CQG. Every weak$^*$ closed left ideal $I\subseteq M(\G)$ is of the form
    $$I = I^u(E) = \{\mu\in M(\G) : \pi(\mu)(E_\pi) = 0 ~ \text{for all} ~\pi\in Irr(\G)\}$$
    where $E = (E_\pi)_{\pi\in Irr(\G)}$ is a closed quantum subset of $\widehat{\G}$ and $\pi : M(\G)\to \fB(\fH_\pi)$ is the natural extension of $\pi\in Irr(\G)$ to $M(\G)$.
\end{itemize}
Another fundamental problem in Banach algebras is describing their idempotents. The idempotents of $B(G)$ for locally compact $G$ were completely characterized as the characteristic functions on sets in the sigma ring generated by clopen subgroups of $G$ by the Cohen-Host idempotent theorem (see \cite{KL19}). Other than in the abelian case, the classification of idempotents in $M(G)$ remains an open problem, even for $SO(3)$. On the other hand, the idempotent states in $B(G)$ are exactly the characteristic functions on clopen subgroups of $G$ (see \cite{IS05}) and the idempotent states in $M(G)$ are exactly the Haar measures coming from compact subgroups of $G$ (see \cite{K58}). The idempotent states in $M(\G)$, where $\G  = U_q(2), SU_q(2)$, and $SO_q(3)$ were completely classified in \cite{FST13}.

For a CQG $\G$, an intimately related problem is the determination of the closed left ideals in $L^1(\G)$ that admit bounded right approximate identities. When $\G$ is coamenable, there is a one-to-one correspondence between idempotents in $M(\G)$ and closed left ideals of $L^1(\G)$ that have bounded right approximate identities (see \cite{NM19}). For a (possibly non-amenable) discrete group $\Gamma$, the ideals in $A(\Gamma)$ with bounded approximate identities were completely characterized by Forrest in \cite{For}. A more specific result in \cite{For} is that an ideal of the form $I(\Lambda)$ for some subgroup $\Lambda$ has a bounded approximate identity if and only if $\Gamma$ is amenable. In other words, if $\Gamma$ is amenable, then every such $I(\Lambda)$ has a bounded approximate identity and when $\Gamma$ is non-amenable, no such $I(\Lambda)$ has a bounded approximate identity. We point out also that for any $s\in \Gamma$, $I(s\Lambda)$ has a bounded approximate identity if and only if $I(\Lambda)$ does, and thus this characterization easily applies to cosets of subgroups as well.

A compact quasi-subgroup is a von Neumann subalgebra of $L^\infty(\G)$ that corresponds to an idempotent state in the universal measure algebra $M^u(\G)$ in a sense that generalizes the identifications in the above paragraph between subgroups of discrete / compact groups (denoted $\Gamma$ / $G$) and idempotent states in $B(\Gamma)$ / $M(G)$ respectively (see Section $4.1$). For example, if $\Gamma$ is a discrete group, the compact quasi-subgroups are the subalgebras $VN(\Lambda)\subseteq VN(\Gamma)$ where $\Lambda$ is a subgroup of $\Gamma$. Then
$$I(\Lambda) = VN(\Lambda)_\perp = \{u\in A(\Gamma) : u(x) = 0 ~ \text{for all} ~ x\in VN(\Lambda)\}.$$
We make progress towards generalizing Forrest's result by proving the following.
\begin{itemize}
    \item Let $\G$ be a CQG and $N \subseteq L^\infty(\G)$ be a compact quasi-subgroup with associated idempotent state $\omega\in M^r(\G)$, where $M^r(\G)$ is the reduced measure algebra. Then $J^1(N) := N_\perp$ has a bounded right approximate identity if and only if $\G$ is coamenable.
\end{itemize}
We point out that $H$ is amenable if and only if $1_H\in B_r(G)$, where $B_r(G)$ is the reduced Fourier-Stieltjes algebra. So, our result is a generalization of Forrest's result applied to amenable subgroups of $G$.

Section $2$ will comprise the preliminaries for locally compact quantum groups where we will in particular recall the theory behind  closed quantum subgroups and more generally, invariant subspaces.

In Section $3$, we will develop the notion of a hull $E$ of a closed left ideal $I\subseteq L^1(\G)$ and then will classify the compact quantum groups such that $\overline{j(E)} = I = I(E)$, for each hull $E$, in terms of Ditkin's property at infinity (or property left $D_\infty$, a property which has recently achieved a new characterization \cite{A20}), (see Theorem \ref{StrcClosedLeftIdeals}). In particular, we can describe the closed left ideals of compact quantum groups whose dual has the approximation property. Then we will show White's techniques \cite{Wh18} for classifying the weak* closed left ideals of the measure algebra of a compact group extend to the setting of coamenable compact quantum groups (see Theorem \ref{wk*IdealsMeasures}). We will conclude the section with a brief discussion of property left $D_\infty$ and provide examples of CQGs which are weakly amenable and consequently have property left $D_\infty$.

Finally, in Section $4$ we study the closed left ideals of $L^1(\G)$ which admit a brai, with special emphasis on the preannihilator space $J^1(N)$ of a compact quasi--subgroup $N$ (the natural quantum analogue of a closed subgroup of a compact group). We also study the associated weak$^*$ closed left ideal $J^u(N)$ in $M^u(\G)$ and in a certain case, show $J^u(N) = \overline{J^1(N)}^{wk*}$ if and only if $\G$ is coamenable if and only if $J^1(N)$ admits a bounded right approximate identity (see Theorems \ref{Wk*IdealsBrais} and \ref{CoamenUniv}, and Corollary \ref{CoamenBAIIdeal}). We conclude the section by showing whenever $N\neq X = Nx$ for $x\in Gr(\G)$, that $X_\perp$ possesses a bounded approximate identity if and only if $\G$ is coamenable (see Theorem \ref{CoamenIdeal}). In this context, we think of $X$ as being a ``quantum coset'' of the compact quasi--subgroup $N$. We end by illustrating these last results on discrete crossed products equipped with the structure of a compact quantum group.
\section{Preliminaries}
\subsection{Locally Compact Quantum Groups}
The notion of a quantum group we will be using in this paper is the one developed by Kusterman and Vaes \cite{KV00}. We use the references \cite{KV03} and \cite{Tim}. A {\bf locally compact quantum group} (LCQG) $\G$ is a quadruple $(L^\infty(\G), \Delta_\G, \psi_L, \psi_R)$ where $L^\infty(\G)$ is a von Neumann algebra; $\Delta_\G : L^\infty(\G) \to L^\infty(\G)\overline{\otimes}L^\infty(\G)$ an injective normal unital $*$--homomorphism satisfying $(\Delta_\G\otimes\id)\Delta_\G = (\id\otimes\Delta_\G)\Delta_\G$ (coassociativity); and $\psi_L$ and $\psi_R$ are normal semifinite faithful weights on $L^\infty(\G)$ satisfying
$$\psi_L(f\otimes \id)\Delta_\G(x) =f(1)\psi_L(x), ~ f\in L^\infty(\G)_*, x\in L^\infty(\G)_{\psi_L} ~ \text{(left invariance)}$$
and
$$\psi_R(\id\otimes f)\Delta_\G(x) = f(1)\psi_R(x), ~f\in L^\infty(\G)_*, x\in L^\infty(\G)_{\psi_R} ~ \text{(right invariance)},$$
where $L^\infty(\G)_{\psi_L}$ and $L^\infty(\G)_{\psi_R}$ are the set of integrable elements of $L^\infty(\G)$ with respect to $\psi_L$ and $\psi_R$ respectively. We call $\Delta_\G$ the {\bf co--product} and $\psi_L$ and $\psi_R$ the {\bf left and right Haar weights} respectively, of $\G$.

Using $\psi_L$, we can build a GNS Hilbert space $L^2(\G)$ with representation $L^\infty(\G)\subseteq \fB(L^2(\G))$. There exists a unitary $W_\G\in L^\infty(\G)\overline{\otimes}\fB(L^2(\G))$ such that $\Delta_\G(x) = W_\G^*(1\otimes x)W_\G$. The unitaries $W_\G$ are known as the {\bf left fundamental unitaries} of $\G$. The predual $L^1(\G) := L^\infty(\G)_*$ is a Banach algebra with respect to the product $f\times g\mapsto f*g:= (f\otimes g)\Delta_\G$ known as {\bf convolution}. This naturally provides us with left and right dual actions on $L^\infty(\G)$, realized by the equations
$$f*x = (\id\otimes f)\Delta_\G(x) ~ \text{and} ~ x*f = (f\otimes\id)\Delta_\G(x).$$
Unfortunately, $L^1(\G)$ is not generally a $*$--algebra. There is, however, a dense $*$--subalgebra we can build. The {\bf antipode} $S_\G : D(S)\to L^\infty(\G)$ is an unbounded linear anti--automorphism, (where $D(S)$ denotes the weak$^*$ dense domain of $S$), satisfying the identity $(\id\otimes S)(W_\G) = W_\G^*$. Then, we define
$$L^1_\#(\G) := \{ f\in L^1(\G) : ~\text{there exists} ~ g\in L^1(\G) ~\text{such that}~ g(x) = \overline{f(S_\G(x))}, ~x\in D(S)\}$$
is an involutive algebra that is dense in $L^1(\G)$ with involution $f\mapsto \overline{f\circ S}$, (see \cite{K01}). Note that the {\bf unitary antipode} $R_\G$ is the unitary in the polar decomposition of $S_\G$.

A {\bf co--representation operator} is an element $U\in L^\infty(\G)\overline{\otimes}\fB(\fH)$, for some Hilbert space $\fH$, such that $(\Delta_\G\otimes\id)(U) = U_{13}U_{23}$ where $U_{23} = 1\otimes U$ and $U_{13} = (\Sigma\otimes\id)U_{23}(\Sigma\otimes\id)$, and $\Sigma$ is the flip map. Co--representation operators $U^\pi$ are in one--to--one correspondence with representations $\pi : L^1(\G)\to \fB(\fH_\pi)$ with which the relationship is realized by setting $\pi(f) = (f\otimes\id)(U^\pi)$. The unitary co--representation operators correspond to representations that restrict to a $*$--representation on $L^1_\#(\G)$. We will simply refer to such representations as $*$ --representations.

Representations $\pi$ and $\rho$ are unitarily equivalent if there exists a unitary $U: \fH_\rho\to \fH_\pi$ such that
$$(1\otimes U^*)U^\pi(1\otimes U) = U^\rho.$$
Whenever we have a representation $\pi$, we are choosing a representative from the equivalence class of unitarily equivalence representations. We call a co--representation operator $U^\pi$ {\bf irreducible} if the corresponding representation $\pi$ is irreducible. We will use the notation $Irr(\G)$ to denote a family of irreducible $*$--representations on $L^1(\G)$, each chosen from distinct equivalence classes of $*$-representations with respect to unitary equivalence.

The fundamental unitary $W_\G$ is a unitary co--representation operator and the corresponding $*$--representation $\lambda_\G$ is called the {\bf left regular representation} of $\G$.

The von Neumann algebra $\overline{\lambda_\G(L^1_\#(\G))}^{wk*}$ has the structure of a LCQG, which in particular, has a co--product unitarily implemented by $W_{\widehat{\G}} = \Sigma W_\G^* \Sigma$. The underlying LCQG, $\widehat{\G}$, is called the {\bf dual LCQG} of $\G$. It turns out that $L^\infty(\G) = \overline{\lambda_{\widehat{\G}}(L^1_\#(\widehat{\G}))}^{wk*}$, which we view as being Pontryagin duality at the level of LCQGs.

There is also $C^*$--algebraic framework underlying the theory of LCQGs. The {\bf reduced} $C^*$--algebra is $C^r_0(\G) := \overline{\lambda_{\widehat{\G}}(L^1_\#(\widehat{\G}))}^{||\cdot||_r}$. The map $a\mapsto \Delta^r(a) := W_\G^*(1\otimes a)W_\G$ defines a non--degenerate $*$--homomorphism $C^r_0(\G)\to M(C^r_0(\G)\otimes_{min} C^r_0(\G))$ satisfying coassociativity and is known as the co--product of $C^r_0(\G)$. By building a universal $*$--representation $\varpi_{\widehat{\G}}$, we can build the {\bf universal} $C^*$--algebra $C^u_0(\G) := \overline{\varpi_{\widehat{\G}}(L^1_\#(\G))}^{||\cdot||_u}$ which also comes equipped with a co--product $\Delta^u_\G$. The {\bf counit} is a character $\epsilon^u_\G : C^u_0(\G)\to\C$ satisfying $(\id\otimes\epsilon_\G^u)\Delta^u_\G = \id = (\epsilon_\G^u\otimes\id)\Delta^u_\G$. There is a surjective $*$--homomorphism $\Gamma_\G : C^u_0(\G)\to C^r_0(\G)$ such that $(\Gamma_\G\otimes\Gamma_\G)\Delta^u_\G = \Delta^r_\G\circ\Gamma_\G$ which is known as the {\bf reducing morphism}.

The dual $M^u(\G) := C^u_0(\G)^*$ is a unital Banach algebra with respect to the product $\mu\times\nu\mapsto \mu*\nu:= (\mu\otimes\nu)\Delta^u$, which again, is called {\bf convolution}, and unit $\epsilon^u_\G$. We call $M^u(\G)$ the {\bf universal measure algebra} of $\G$. Similarly $M^r(\G) := C^r_0(\G)^*$ is Banach algebra known as the {\bf reduced measure algebra}. We remark that $M^r(\G) = \overline{L^1(\G)}^{wk*}$ and the adjoint $\Gamma_\G^* : M^r(\G)\to M^u(\G)$ gives us a completely isometric embedding, and furthermore, $M^u(\G)$ contains $L^1(\G)$ as an ideal through this embedding. Whenever $M^r(\G)$ turns out to be unital (so that $C^r_0(\G)$ admits a ``reduced'' counit), we denote this unit by $\epsilon^r_\G$.
\begin{rem}
    The examples where $L^\infty(\G)$ is commutative are the locally compact groups. We call the dual of a locally compact group a {\bf locally compact co--group} and we use the notation $C^r_0(\widehat{G}) = C^*_\lambda(G)$, $C^u_0(\widehat{G}) = C^*(G)$, $L^1(\widehat{G}) = A(G)$, $C^u_0(\widehat{G})^* = B(G)$, and $L^\infty(\widehat{G}) = VN(G)$, i.e, $\widehat{G} = (VN(G), \Delta_{\widehat{G}}, \varphi)$ where $\Delta_{\widehat{G}}(\lambda_G(s)) = \lambda_G(s)\otimes\lambda_G(s)$ and $\varphi = \psi_L = \psi_R$ is the Plancherel weight \cite[Chapter IV, Section 3]{Tak03}. The Banach algebra $A(G)$ is called the Fourier algebra and $B(G)$ the Fourier--Stieltjes algebra. The locally compact co--groups comprise the examples where $L^1(\G)$ is commutative, i.e., $L^\infty(\G)$ is cocommutative.
\end{rem}
\subsection{Compact Quantum Groups and Fourier Analysis}
The {\bf compact quantum groups (CQGs)} are the LCQGs $\G$ such that $L^1(\widehat{\G})$ is unital, meaning that $C^u_0(\G)$ and $C^r_0(\G)$ are both unital (cf. \cite{Wor} and \cite{R08}). In this case, we follow the custom of writing $C^u(\G)$ and $C^r(\G)$ for the universal and reduced $C^*$--algebras respectively instead. It follows that $\psi_L = \psi_R = h_\G\in L^1(\G)$ is a state called the {\bf Haar state} of $\G$, the irreducible representations are finite dimensional, and every $*$--representation decomposes into a direct sum of irreducibles. Given an irreducible representation $\pi$, we will write $U^\pi = [u_{i,j}^\pi]$ with respect to an orthonormal basis (ONB) $\{e_i^\pi\}$ of $\fH_\pi$. Then,
$$\Pol(\G) := \Span\{ u_{i,j}^\pi : 1\leq i,j\leq n_\pi, \pi\in Irr(\G)\}$$
is a $*$--algebra that identifies with a norm dense $*$--subalgebra of $C^u(\G)$ and $C^r(\G)$, and a weak$^*$ dense $*$--subalgebra of $L^\infty(\G)$. It is actually a Hopf $*$--algebra with co--product $\Delta_\G := \Delta_\G|_{\Pol(\G)} \Pol(\G)\to \Pol(\G)\otimes \Pol(\G)$, counit $\epsilon_\G := \epsilon^u_\G|_{\Pol(\G)}$, and antipode $S_\G|_{L^\infty(\G)}$, which satisfy
\begin{align*}
    \Delta_\G(u_{i,j}^\pi) = \sum^{n_\pi}_{t=1}u_{i,t}^\pi\otimes u_{t,j}^\pi, ~ \epsilon_\G(u_{i,j}^\pi) = \delta_{i,j}, ~ S_\G(u_{i,j}^\pi) = (u_{j,i}^\pi)^*
\end{align*}
Vital to the theory of CQGs is Fourier analysis, which is essentially analysis of Fourier transform, which is the contractive embedding $L^1(\G)\to C^r_0(\widehat{\G})$ induced by the left regular representation. So, a discussion about Fourier analysis involves a discussion about the duals of CQGs. We recommend \cite[Section 2]{Wang16} as reference on the following discussion.

A {\bf discrete quantum group (DQG)} is the Pontryagin dual of a CQG. We briefly outline how this works here. It turns out we have the decomposition $\lambda_\G = \oplus_{\pi\in Irr(\G)}\pi = \varpi_\G$ (cf.\cite{Wor}). From this, we have the algebra decompositions
\begin{align*}
    c_0(\widehat{\G}) &:= \bigoplus_{\pi\in Irr(\G)}^{c_0}M_{n_\pi} = C^r_0(\widehat{\G}) = C^u_0(\widehat{\G})
    \\
    \ell^\infty(\widehat{\G}) &:= \bigoplus_{\pi\in Irr(\G)}^{\ell^\infty} M_{n_\pi} = L^\infty(\widehat{\G}).
\end{align*}
In particular, the Fourier transform is the map
$$f\mapsto \lambda_\G(f) = \oplus_{\pi\in Irr(\G)}\pi(f) \in c_0(\widehat{\G})$$
contractively embedding $L^1(\G)$ into $c_0(\widehat{\G})$ as a dense subspace. The left and right Haar weights on $\widehat{\G}$ are realized as the direct sums:
\begin{align*}
    \psi_L = \bigoplus_{\pi\in Irr(\G)}\tr(F_\pi)\tr(F_\pi\cdot), ~ \psi_R = \bigoplus_{\pi\in Irr(\G)}\tr(F_\pi)\tr(F_\pi^{-1}\cdot)
\end{align*}
for some positive, invertible matrices $F_\pi\in M_{n_\pi}$, uniquely determined with the normalization $\tr(F_\pi) = \tr(F_\pi^{-1})>0$. For each $\pi\in Irr(\G)$, an ONB may be chosen so that $F_\pi$ is diagonal (see \cite{Tim} for general information and \cite{D10} for why they can be taken to be diagonal).

We also have a linear contraction $L^\infty(\G)\to L^1(\G)$ given by the map $x\mapsto \widehat{x} := h_\G\cdot x$, where $(h_\G\cdot x)(y) = h(xy)$. Then, if we let $e_{i,j}^\pi\in M_{n_\pi}$ be the matrix unit with $1$ in the $(i,j)$ entry,
$$e_{i,j}^\pi = \sum_{k=1}^{n_\pi}\frac{1}{\tr(F_\pi)}(F_\pi)_{i,k}^{-1}\widehat{(u_{k,j}^\pi)^*}.$$
In particular, $\lambda_\G(\widehat{\Pol(\G)}) = \bigoplus_{\pi\in Irr(\G)} M_{n_\pi} =: c_{00}(\widehat{\G})$.

We have the following convolution formula that will be of use to us: for $f\in L^1(\G)$, $x\in L^\infty(\G)$, and $\pi\in Irr(\G)$,
$$\pi(\widehat{f*x}) = \pi(\hat{x})\pi(f\circ S_\G).$$
See Remark \ref{antipode} for the apparent domain issue in the above equation.
\subsection{Invariant Subspaces, Ideals, and Quotients}
In this work, we are primarily interested in the left ideals of the $L^1$-algebras and measure algebras of CQGs. By exploiting the duality between $L^1(\G)$ and $L^\infty(\G)$, we identify subspaces of $L^\infty(\G)$ that correspond to left ideals in $L^1(\G)$. We also do this for the measure algebras of $\G$.
\begin{defn}
    For a LCQG $\G$, we say a subset $E\subseteq L^\infty(\G)$ is {\bf right invariant} if $E*f\subseteq E$ for all $f\in L^1(\G)$. Likewise, we say a $E\subseteq C^u(\G)$ (or $C^r(\G)$) is {\bf right invariant} if $E*\mu\subseteq E$ for all $\mu\in M^u(\G)$ (or $M^r(\G)$ resp.) For a CQG $\G$ we say a subset $E\subseteq \Pol(\G)$ is {\bf right invariant} if $E*\mu\subseteq E$ for all $\mu\in \Pol(\G)^*$. In all cases, we define left invariance similarly and $E$ is invariant if $E$ is both left and right invariant. If $X\subseteq L^\infty(\G)$ is a left invariant weak$^*$ closed subspace, then we will write $$X\trianglelefteq_l L^\infty(\G)$$ and likewise
    $X\trianglelefteq_r L^\infty(\G)$ if it is right invariant instead.
\end{defn}
\begin{rem}
    It is straightforward seeing that a von Neumann subalgebra $N\subseteq L^\infty(\G)$ is right invariant if and only if
    $$\Delta_\G(N)\subseteq L^\infty(\G)\overline{\otimes}N.$$
\end{rem}
From the bipolar theorem, it is easy to see that if $X\subseteq L^\infty(\G)$ is a right invariant subspace, then $X_\perp$ is a closed left ideal in $L^1(\G)$, and whenever $I$ is a left ideal in $L^1(\G)$, $I^\perp$ is a weak* closed right invariant subspace of $L^\infty(\G)$. We package this observation and similar ones into the following.
\begin{prop}
    The norm closed right ideals of $L^1(\G)$ are in one to one correspondence with the weak* closed right invariant subspaces of $L^\infty(\G)$ via
    $$L^1(\G) \trianglerighteq_l I\iff I^\perp \trianglelefteq_r L^\infty(\G),$$
    and the weak* closed right ideals of $M^r(\G)$ and $M^u(\G)$ are in one to one correspondence with the the norm closed left invariance subspaces of  $C^r_0(\G)$ and $C^u_0(\G)$ via
    $$C^r_0(\G), C^u_0(\G) \trianglerighteq_l X\iff X^\perp \trianglelefteq_r M^r(\G), M^u(\G).$$
    For right ideals and left invariant subpsaces, we will use the notation $\trianglelefteq_r$ and $\trianglerighteq_l$ respectively. 
\end{prop}
Of special interest are the invariant subalgebras. If $N$ is a right invariant von Neumann subalgebra of $L^\infty(\G)$, then we call $N$ a {\bf right coideal}. When $G$ is a locally compact group, the right coideals of $VN(G)$ are exactly of the form $VN(H)$ where $H$ is a closed subgroup of $G$, and are invariant. Similarly, the right coideals of $L^\infty(G)$ are exactly of the form $L^\infty(G/H)\subseteq L^\infty(G)$ using the identification of $L^\infty(G/H)$ with the functions that are constant on cosets of $H$ (see \cite{KL19}). So, the right coideals of a LCQG offer a quantum analogue of the algebra of functions that are constant on cosets of a subgroup.

In Section $4.1$ we focus on the left ideals corresponding to a certain subclass of right coideals of a CQG. In this case we will denote $J^1(N) = N_\perp$. Notice that $N_* \cong L^1(\G)/J^1(N)$ as Banach spaces. Here is an alternative view on the relationship between $N$ and $J^1(N)$ afforded by weak* closures. Using weak*-weak* continuity of the inclusion $N\subseteq L^\infty(\G)$, find the preadjoint
$$T_N : L^1(\G) \to N_*.$$
Then $J^1(N) = \ker(T_N)$. Notice that whenever $N$ is invariant, $J^1(N)$ is a two sided ideal, and so $T_N$ is an algebraic homomorphism, meaning $N_*$ has a Banach algebra structure inherited from the quotient algebra $L^1(\G)/J^1(N)$.

\subsection{Multipliers and Approximation Properties}
We say a LCQG $\G$ is {\bf coamenable} if $L^1(\G)$ has a bai and is {\bf amenable} if $L^\infty(\G)$ possesses a left invariant state, that is, a state $m\in L^\infty(\G)^*$ such that $f(\id\otimes m)\Delta_\G(x) = f(1)m(x)$ for all $f\in L^1(\G)$.
\begin{rem}
    The Haar state for a CQG is, in particular, a left invariant state, meaning CQGs are amenable. Likewise, DQGs have unital $L^1$--algebras and hence are coamenable. It also turns out that a CQG is coamenable if and only if its dual is amenable \cite{Ruan, Bedos1, Bedos2, Bedos3, Toma}, generalizing Leptin's theorem for locally compact groups. On the other hand, while generally for LCQGs it is not too difficult to show coamenability of a LCQG implies amenability of its dual, the converse remains an open problem.
\end{rem}
\begin{rem}
    We have that $\G$ is coamenable if and only if $C^u_0(\G)\cong C^r_0(\G)$ and so in this case, we will simply write $C(\G)$ and similarly $M(\G)$ for the (in this case) distinguished measure algebra.
\end{rem}
We will also be interested in weakened versions of amenability. In lieu of the duality between coamenability and amenability (for CQG/DQGs), the natural choice is to weaken boundedness of a left or right bai (blai or brai) in $L^1(\widehat{G})$. To discuss relevant versions of this in the literature, we must first discuss multipliers and completely bounded multipliers. We recommend \cite{B16, HNR11} as references for the following discussions.
\begin{defn}
    A {\bf left multiplier} of a Banach algebra $A$ is a bounded linear right $A$--module map $m : A\to A$. We denote the left multipliers on $A$ by $M^l(A)$.  We denote the {\bf completely bounded left multipliers} by $M^l_{cb}(A) : =  M^l(A)\cap \mathcal{CB}(A)$.
\end{defn}
\begin{rem}
    Note that the elements $m\in M^l_{cb}(A)$ are exactly those for which the Banach space adjoint $m^* : A^*\to A^*$ is completely bounded.
\end{rem}
We note $M^l(A)$ is a Banach algebra, viewed as a subalgebra of $\fB(A)$, which has $A$ embedded contractively as an ideal via the map $a\mapsto m_a$ where $m_a(b) = ab$ for $b\in A$, and we will denote the Banach space adjoint $M_a := m_a^* : A^*\to A^*$. Similarly, $M^l_{cb}(A)$ is a c.c. Banach algebra, viewed as an operator subspace of $\mathcal{CB}(A)$, which has $A$ embedded into $M^l_{cb}(A)$ completely contractively. For a LCQG $\G$, because $M^u(\G)$ contains $L^1(\G)$ as an ideal, we get that $M^u(\G)$ embeds completely contractively into $M_{cb}^l(L^1(\G))$ via the map $\mu\mapsto m_\mu$ where $m_\mu(f) = \mu*f$ for $f\in L^1(\G)$ and again we denote the adjoint by $M_\mu$. Note also that $M_{\epsilon_\G^u} = \id$.

The {\bf double centralizers} of a Banach algebra $A$ are pairs $(L,R)$ of left and right multipliers $L\in M^l(A)$ and $R\in M^r(A)$ satisfying $aL(b) = R(a)b$. We denote the double centralizers by $M(A)$, which also turns out to be a Banach algebra, and has $A$ contractively embedded as an ideal via the map $a\mapsto (l_a, r_a)$ where $l_a(b) = ab$ and $r_a(b) = ba$. There is also a contractive embedding $M(A)\subseteq M^l(A)$. Similarly, we define {\bf completely bounded double centralizers} of a c.c. Banach algebra $A$, which are double centralizers whose associated bounded linear maps and completely bounded. We denote the completely bounded double centralizers by $M_{cb}(A)$. Similarly, we have completely contractive embeddings $A\subseteq M_{cb}(A)\subseteq M_{cb}^l(A)$. For LCQGs, we have $M^u(\G)\subseteq M_{cb}(L^1(\G))$.
\begin{rem}
    Whenever $\G$ is coamenable, it is the case that $M(L^1(\G)) = M^l(L^1(\G)) = M(\G)$. For locally compact co--groups, this property characterizes amenability. That $M(A(G)) = B(G)$ implies amenability for discrete $G$ is due to \cite{N82} (and generally is due to Losert \cite{L84}), and Losert extended this to the case of $M_{cb}A(G)$ in an unpublished manuscript. For discrete G, however, see \cite{BF84}. For a LCQG $\G$ in general, we also have that the completely isometric equalities $M^l_{cb}(L^1(\G)) = M^r(\G) = M^u(\G)$ characterizes coamenability (cf. \cite{HNR11}).
\end{rem}
A first weakening, then, would be to loosen the boundedness criterion of the bai.
\begin{defn}
    We say a LCQG $\G$ is {\bf weakly amenable} if there exists a net $(f_i)\subseteq L^1(\widehat{\G})$ such that $f_i*f\to f$ for all $f\in L^1(\widehat{\G})$ and $\sup_i ||f_i||_{M_{cb}^l} < \infty$.
\end{defn}
There is another relevant, even weaker version of amenability.
\begin{defn}
    We say a LCQG $\G$ has the {\bf approximation property (AP)} if there exists a net $(f_i)\subseteq L^1(\widehat{\G})$ such that $M_{f_i}\to \id_{L^\infty(\widehat{\G})}$ in the stable point weak$^*$ topology of $\mathcal{CB}(L^\infty(\widehat{\G}))$, by which we mean, for a separable Hilbert space $\fH$, we have $$\varphi(M_{e_i}\otimes\id)(x)\to \varphi(x)$$ for all $x\in L^\infty(\widehat{\G})\overline{\otimes}\fB(\fH)$ and $\varphi\in L^1(\widehat{\G})\widehat{\otimes}\fB(\fH)_*$.
\end{defn}
\begin{rem}
    The definitions of weak amenability and the AP for quantum groups are generalizations of the namesake definitions in the case where $\G = G$ is a locally compact group. They are meant to be properties that are weaker than amenability of $G$. Leptin's theorem \cite{L68} states that $G$ is amenable if and only if $L^1(\widehat{G})$ has a bai, i.e., $\widehat{G}$ is coamenable. Weak amenability and the AP are properties where $L^1(\widehat{G})$ admits an approximate identity that has a restriction on its asymptotic behaviour, but may not be bounded with respect to the norm on $L^1(\widehat{G})$. For example, finitely generated free groups are examples of non-amenable discrete groups that are weakly amenable.
\end{rem}
We have the following description of the AP, which we point out implies $L^1(\G)$ has a left approximate identity when we let $H = \{e\}$, where $H$ is as introduced in the following proposition. Note that when $H$ is compact, $1_H$ is the identity element in $A(H)$.
\begin{prop}\label{APai}
    For a LCQG $\G$, TFAE:
    \begin{enumerate}
        \item $\widehat{\G}$ has the AP;
        \item for every compact group $H$, there is a net $(f_i)\subseteq L^1(\G)$ such that
        $$||(f_i\otimes 1_H)*g - g||_1\to 0$$
        for all $g\in L^1(\G)\widehat{\otimes}A(H)$;
        \item and there is a net $(f_i)\subseteq L^1(\G)$ such that
        $$||(f_i\otimes 1_{SU(2)})*g - g||_1\to 0$$
        for all $g\in L^1(\G)\widehat{\otimes}A(SU(2))$.
    \end{enumerate}
\end{prop}
\begin{proof}
    The proof follows verbatim the proof of \cite[Theorem 1.11]{KH94}: the techniques are entirely functional analytic and pass directly to LCQGs.
\end{proof}
Note that the map $M_{cb}(L^1(\G))\ni m\mapsto m^* \in \mathcal{CB}_{L^1(\G)}^\sigma(L^\infty(\G))$ is a completely isometric isomorphism, where $\mathcal{CB}^\sigma_{L^1(\G)}(L^\infty(\G))$ denotes the normal completely bounded right $L^1(\G)$-module maps on $L^\infty(\G)$. Crann pointed out in \cite[Proposition 3.2]{C17(1)} that the work of Kraus and Ruan \cite[Theorem 2.2]{KR99} extends directly from Kac algebras to LCQGs so that we obtain a predual
$$M^l_{cb}(L^1(\G))_* =Q_{cb}^l(L^1(\G)) = \{\omega_{A,\varphi} : A\in C^r_0(\G)\otimes_{min}\fK(\fH), \varphi\in L^1(\G)\widehat{\otimes}\mathcal{T}(\fH)\}$$
where $\omega_{A,\varphi}(M) = \varphi(M\otimes\id)(A)$ for $M\in M_{cb}^l(L^1(\G))$, $\fH$ is a separable Hilbert space, and $\mathcal{T}(\fH) = \fB(\fH)_*$ is the space of trace class operators.

For $T\in L^\infty(\G)\overline{\otimes}\fB(\fH)$ and $\varphi\in L^1(\G)\widehat{\otimes}\mathcal{T}(\fH)$, we similarly use $\omega_{T,\varphi}$ to denote the functional satisfying $\omega_{T,\varphi}(M) = \varphi(M\otimes\id)(T)$. Then, the argument used by Kraus and Ruan in \cite[Proposition 5.2]{KR99} extends directly to CQGs to give us the following.
\begin{prop}\label{CQGsRepresentationQcb}
    Let $\G$ be a CQG. Then $\omega_{T,\varphi}\in Q^l_{cb}(L^1(\G))$ for all $T\in L^\infty(\G)\overline{\otimes}\fB(\fH)$ and $\varphi\in L^1(\G)\widehat{\otimes}\mathcal{T}(\fH)$.
\end{prop}
With Proposition \ref{CQGsRepresentationQcb} in hand, a proof verbatim to the proof in \cite[Theorem 5.4]{KR99} shows the following, which allows us to view the AP of the duals of CQGs within our framework of weakening coamenability. For convenience, we will supply a proof.
\begin{prop}\label{APweak*bai}
    We have that a DQG $\G$ has the AP if and only if there exists a net $(e_i)\subseteq L^1(\widehat{\G})$ such that $e_i\to \epsilon_{\widehat{\G}}^u$ in the $\sigma(M_{cb}^l(L^1(\widehat{\G})), Q_{cb}(L^1(\widehat{\G})))$ topology on $M_{cb}^l(L^1(\widehat{\G}))$.
\end{prop}
\begin{proof}
    Suppose $\G$ has the AP. Let $(e_i)\subseteq L^1(\widehat{\G})$ be a net such that $M^l_{e_i}$ converges in the stable point weak$^*$ topology to $\id$. In particular, for $A\in C^r(\G)\otimes_{min}\fK(\fH)$ and $\varphi\in A(\G)\widehat{\otimes}\mathcal{T}(\fH)$
    \begin{align*}
        \omega_{A,\varphi}(M_{e_i}) &= \varphi(M_{e_i}\otimes \id)(A)\to \varphi(A) = \omega_{A,\varphi}(M_{\epsilon_{\widehat{\G}^u}})
    \end{align*}
    which says exactly that $e_i\to \epsilon_{\widehat{\G}}^u$ in the weak$^*$ topology on $M^l_{cb}(L^1(\widehat{\G}))$.\\
    \\
    Conversely, suppose $e_i\to\epsilon_{\widehat{\G}}^u$ in the weak$^*$ topology on $M^l_{cb}(L^1(\widehat{\G}))$. Then, for $T\in VN(\G)\overline{\otimes}\fB(\fH)$ and $\varphi\in A(\G)\widehat{\otimes}\mathcal{T}(\fH)$, using Proposition \ref{CQGsRepresentationQcb}, we have
    \begin{align*}
        \varphi(M_{e_i}\otimes\id)(T) = \omega_{T,\varphi}(M_{e_i}^l) &\to \omega_{T,\varphi}(M_{\epsilon_{\widehat{\G}}^u}) = \varphi(T)
    \end{align*}
    which says exactly that $M_{e_i}\to \id$ in the stable point weak$^*$ topology.
\end{proof}
\begin{rem}
    We are unaware of a version of Proposition \ref{APweak*bai} for general LCQGs. To prove the result, we would require a general version of Proposition \ref{CQGsRepresentationQcb}, however, the proof makes essential use of the underlying Hopf algebras of CQGs and does not clearly extend to general LCQGs.
\end{rem}
An immediate observation from Proposition \ref{APweak*bai} is the following.
\begin{cor}
    A weakly amenable DQG has the AP.
\end{cor}
We should note that we believe that the results of this subsection are not new. We included them and their justifications because we have not been able to find a reference and we will be using the fact that quantum groups with the AP have an approximate identity, and that weak amenability implies the AP. 
\section{Structure of Left Ideals}
\subsection{Left Ideals of $L^1$--algebras}
For the rest of this paper, $\G$ will always denote a CQG unless otherwise specified. This section contains the most fundamental results of our paper. A structure theorem for the left ideals of $L^1(\G)$ is obtained, which, as we will make explicit in Section $3.3$, allows us to describe the closed left ideals of $L^1(\G)$ for many CQGs $\G$, including those that have the approximation property. In lieu of the previous section, this includes the weakly amenable CQGs.

Recall that for $x\in \Pol(\G)$, $\hat{x} = h\cdot x$. So, fix a CQG $\G$ and let
$$E = (E_\pi)_{\pi\in Irr(\G)}$$
where each $E_\pi\subseteq\fH_\pi$ is a subspace (possibly trivial or all of $\fH_\pi$). We will write
$$I(E) = \{ f\in L^1(\G) : \pi(f)(E_\pi) = 0, ~E_\pi\in E\}$$
and
$$j(E) = I(E)\cap \widehat{\Pol(\G)} = I(E)\cap \lambda_\G^{-1}(c_{00}(\widehat{\G})),$$
It is easy to check $j(E)$ is a left ideal and $I(E)$ a closed left ideal in $L^1(\G)$. If $I$ is an ideal such that $j(E)\subseteq I\subseteq I(E)$ then the set $E$ will be referred to as a {\bf hull} of $I$.
\begin{rem}
    It should be addressed that $I(E)$, and hence $j(E)$, is independent of the choice of representatives in $Irr(\G)$ (where the subspaces $E_\pi$ are chosen up to isomorphism). Indeed, suppose $\pi$ is unitarily equivalent to $\rho$, and write $(1\otimes U^*)U^\pi(1\otimes U) = U^\rho$. Note that the unitary $U$ is a Hilbert space isomorphism $U : \fH_\rho\to \fH_\pi$. Then $\pi(f)\xi = U\rho(f)U^*\xi$ shows $\pi(f)\xi = 0$ if and only if $\rho(f)U^*\xi = 0$. In particular, we have $UE_\rho = E_\pi$.
\end{rem}
\begin{defn}
    We say $E$ is a {\bf set of synthesis} if $I(E) = \overline{j(E)}$.
\end{defn}
Before proceeding, we recall the following well-known fact.
\begin{lem}\label{leftIdealsMn}
    Let $I$ be a left ideal in some matrix algebra $M_n$. Then $I = \{ A\in M_n : A(E) = 0\}$ for some subspace $E\subseteq \C^n$.
\end{lem}
The proof of Lemma \ref{leftIdealsMn} can be found, for example, in \cite[Lemma 38.11]{HR63}.
\begin{prop}\label{idealsinL^1}
    Let $\G$ be a CQG and $I\trianglelefteq L^1(\G)$ a left ideal. Then there is a hull $E$ such that $j(E)\subseteq I\subseteq I(E)$.
\end{prop}
\begin{proof}
    We follow the methods used for compact groups in \cite{HR63}. Let $1_\pi$ be the projection onto $M_{n_\pi}$. Then $\pi(f) = 1_\pi\lambda_\G(f)$, and so from density, combined with the fact $M_{n_\pi}$ is finite dimensional, $\pi(L^1(\G)) = M_{n_\pi}$. Consequently, $\pi(I)$ is a left ideal in $M_n$. Then, using Lemma \ref{leftIdealsMn}, we can write $\pi(I) = \{\pi(f)\in M_n : f\in L^1(\G), \pi(f)(E_\pi) = 0\}$ for some subspace $E_\pi\subseteq \fH_\pi$. Let $E = (E_\pi)_{\pi\in Irr(\G)}$, where each $E_\pi\subseteq \fH_\pi$ is the aforementioned subspace for each $\pi$. From here, it is easy to see that $I\subseteq I(E)$.
    
    Before proceeding, we note that for any finite dimensional $*$-representation $\pi$ of $L^1(\G)$, $\pi(j(E)) = \pi(I)$. Indeed, $\pi$ decomposes into a finite direct sum of irreducible representations, so $\pi(I)\subseteq \pi(L^1(\G))\subseteq c_{00}(\widehat{\G})$. Then,
    $$\pi(j(E)) = \pi(I)\cap c_{00}(\widehat{\G}) = \pi(I)$$
    where the first equality follows because $\lambda_\G(j(E)) = \lambda_\G(I)\cap c_{00}(\widehat{\G})$.
    
    Now take $f\in j(E)$, so $\lambda_\G(f) = \oplus_{i=1}^n \pi_i(f)$ for some $\pi_1,\ldots,\pi_n\in Irr(\G)$. Since $\oplus_{i=1}^n \pi_i(j(E)) =\oplus_{i=1}^n \pi_i(I)$, we can find $g\in I$ so that $\oplus_{i=1}^n \pi_i(g) = \lambda_\G(f)$. Set $P = \oplus_{i=1}^n I_{n_{\pi_i}}$, where each $I_{n_{\pi_i}}\in M_{n_{\pi_i}}\subseteq \ell^\infty(\widehat{\G})$ is the identity matrix, and let $e\in L^1(\G)$ be such that $\lambda_\G(e) = P$. Then
    $$\lambda_\G(I)\ni \lambda_\G(e*g) = P\lambda_\G(g) = \bigoplus_{i=1}^n \pi_i(g) = \lambda_\G(f).$$
\end{proof}
The two-sided case is as follows.
\begin{cor}\label{twoidealsinL^1}
   Let $\G$ be a CQG and $I\trianglelefteq L^1(\G)$ an ideal. Then there exists a hull of $\G$, say $E$, such that $j(E)\subseteq I\subseteq I(E)$ where each $E_\pi\in E$ is either $\fH_\pi$ or $\{0\}$.
\end{cor}
\begin{proof}
    Following Proposition \ref{idealsinL^1}, what is left is noticing that each $E_\pi\in E$ must satisfy either $E_\pi = \fH_\pi$ or $E = \{0\}$. Inspecting the proof of Proposition \ref{idealsinL^1}, the result follows because $\pi(I)$ is a two sided ideal $M_{n_\pi}$: we either have $\pi(I) = M_{n_\pi}$ or $\pi(I) = \{0\}$.
\end{proof}
Given the duality between right invariant subspaces of $L^\infty(\G)$ and left ideals of $L^1(\G)$ observed in Section $2.3$, we immediately have that for any weak$^*$ closed right invariant subspace $X$ of $L^\infty(\G)$ that there exists a hull $E$ such that
$$I(E)^\perp \subseteq X\subseteq j(E)^\perp.$$
Then, a natural question is, can we describe $I(E)^\perp$ and $j(E)^\perp$ explicitly in terms of $E$? It turns out we can, and the answer to this question will be important for us when we characterize the CQGs such that every hull is a set of synthesis.

Given a hull $E$, we will denote
$$\Pol(\G,E) = \{ u^\pi_{\xi,\eta} :  \eta\in E_\pi, \xi\in\fH_\pi ~ \pi\in Irr(\G)\},$$
where $u^\pi_{\xi,\eta} = (\id\otimes w_{\eta,\xi})(U^\pi)$ and $w_{\eta,\xi}(T) = \langle T\eta,\xi\rangle$ for $T\in M_{n_\pi}$. Then, we will denote
$$C^r(\G, E) = \overline{\Pol(\G,E)}^{||\cdot||_r}, ~ C^u(\G,E) = \overline{\Pol(\G,E)}^{||\cdot||_u}, ~ \text{and} ~L^\infty(\G,E) = \overline{C^r(\G,E)}^{wk*}.$$
Recall from Section $2.2$ that $\hat{x} = h_\G\cdot x \in L^1(\G)$.
\begin{prop}\label{Annj(E)}
    Let $\G$ be a CQG and $E$ a hull. Then
    \begin{enumerate}
        \item \begin{align*}
        j(E)^\perp &= \overline{\{x\in L^\infty(\G) : \mathrm{im}(\pi(f\circ S_\G))\subseteq \ker(\pi(\widehat{x})), ~ f\in I(E), ~ \pi\in Irr(\G)\}}^{wk*}
        \\
        &= \bigcap_{f\in I(E)} \ker((\mathrm{id}\otimes f)\Delta_\G);
    \end{align*}
    \item and $I(E)^\perp = L^\infty(\G,E) = \overline{\{u^\pi_{\xi,\eta} : \pi(f)\xi\neq 0, ~f\in I(E), \pi\in Irr(\G), \xi\in E_\pi\}}^{wk*}$.
    \end{enumerate}
\end{prop}
\begin{rem}\label{antipode}
    Because the symbol $f\circ S_\G$ is defined only for $f\in L^1_\#(\G)$, an explanation of the notation in Proposition \ref{Annj(E)} is in order. We set
    $$\pi(f\circ S_\G) := (f\otimes \id)(S_\G\otimes \id)U^\pi$$
    which is defined because $S_\G|_{\mathrm{Pol}(\G)} : \Pol(\G)\to \Pol(\G)$ is a bijection.
\end{rem}
\begin{proof}
    $1.$ We will first notice
    \begin{align*}
        &\overline{\{x\in L^\infty(\G) : \mathrm{im}(\pi(f\circ S_\G))\subseteq \ker(\pi(\widehat{x})), ~ f\in I(E), \pi\in Irr(\G)\}}^{wk*}
        \\
        &= \overline{\{x\in L^\infty(\G) : \mathrm{im}(\pi(f\circ S_\G))\subseteq \ker(\pi(\widehat{x})), ~ f\in j(E), \pi\in Irr(\G)\}}^{wk*}
    \end{align*}
    and then will show
    $$j(E)^\perp = \overline{\{x\in L^\infty(\G) : \mathrm{im}(\pi(f\circ S_\G))\subseteq \ker(\pi(\widehat{x})), ~ f\in j(E), \pi\in Irr(\G)\}}^{wk*}.$$
    Accordingly, suppose $x\in L^\infty(\G)$ satisfies $\im(\pi(h\circ S_\G))\subseteq\ker(\pi(\widehat{x}))$ for all $h\in j(E)$.
    
    We will prove that for $f\in I(E)$, we can find $g\in j(E)$ such that $\pi(f\circ S_\G) = \pi(g\circ S_\G)$. Recall that $(S_\G\otimes\id)U^\pi = [S_\G(u_{i,j}^\pi)]_{i,j} = [(u_{j,i}^\pi)^*]_{i,j}$. As shown in the proof of Proposition \ref{idealsinL^1}, we have that $\sigma(I(E)) = \sigma(j(E))$ for every finite dimensional $*$-representation $\sigma$. The condition that $\sigma(f) = \sigma(g)$ is equivalent to having $f(u_{i,j}^\sigma) = g(u_{i,j}^\sigma)$ for every $1\leq i,j\leq n_\sigma$. Since $S_\G|_{\mathrm{Pol(\G)}}$ is a bijection and $$\Pol(\G) = \Span\{u_{i,j}^\pi : 1\leq i,j\leq n_\pi, \pi\in Irr(\G)\}$$
    we have $(u_{j,i})^* = \sum \alpha_{i,j}^{\sigma_k}u^{\sigma_k}_{i,j}$ for some $\sigma_1, \ldots, \sigma_l\in Irr(\G)$, where $l < \infty$ and $\alpha_{i,j}^{\sigma_k}\in \C$. Now choose $g\in j(E)$ such that $\oplus_{k=1}^l\sigma_k(f) = \oplus_{k=1}^l\sigma_k(g)$. Then,
    $$\pi(f\circ S_\G) = [f((u_{j,i}^\pi)^*)]_{i,j} = [g((u_{j,i}^\pi)^*)]_{i,j} = \pi(g\circ S_\G).$$
    So, for $f\in I(E)$, choose $g\in j(E)$ such that $\pi(f\circ S_\G) = \pi(g\circ S_\G)$. Then
    $$\pi(\widehat{x})\pi(f\circ S_\G) = \pi(\widehat{x})\pi(g\circ S_\G) = 0,$$
    as desired. The reverse containment is obvious.
    
    Moving on to the main part of the proof, let $x\in j(E)^\perp$. Then, for $f\in j(E)$,
    $$f*x = (\id\otimes f)\Delta_\G(x) = 0$$
    thanks to right invariance of $j(E)^\perp$. Therefore,
    \begin{align}
        0 = \pi(\widehat{f*x}) = \pi(\widehat{x})\pi(f\circ S_\G), \label{ComputeZero}
    \end{align}
    in other words, $\im(\pi(f\circ S_\G))\subseteq \ker(\pi(\widehat{x}))$.
    
    Conversely, take $x\in L^\infty(\G)$ such that $\im(\pi(f\circ S_\G))\subseteq \ker(\pi(\widehat{x}))$ for all $f\in j(E)$. Following Equation \ref{ComputeZero} in reverse we get $f*x = 0$. Since $f\in \lambda_\G^{-1}(c_{00}(\widehat{\G}))$, we can find $\pi_1,\ldots, \pi_n\in Irr(\G)$ so that $\pi(f) = 0$ if $\pi_i\neq \pi\in Irr(\G)$ for all $i$. Let $\epsilon_{\oplus_{i=1}^n\pi_i} = \lambda^{-1}_\G(1_{\oplus_{i=1}^n M_{n_{\pi_i}}})$ via the identification $\oplus_{i=1}^n M_{n_{\pi_i}}\subseteq c_{00}(\widehat{\G})$. Then
    $$\lambda_\G(\epsilon_{\oplus_{i=1}^n\pi_i}*f) = (\oplus_{i=1}^n\pi_i)(\epsilon_{\oplus_{i=1}^n\pi_i})(\oplus_{i=1}^n\pi_i)(f) = \lambda_\G(f),$$
    so $\epsilon_{\oplus_{i=1}^n\pi_i}*f = f$. Similarly, $f*\epsilon_{\oplus_{i=1}^n\pi_i} = f$. Therefore,
    $$0 = \epsilon_{\oplus_{i=1}^n\pi_i}(f*x) = f(x),$$
    as desired.
    
    Now we justify
    $$j(E)^\perp = \bigcap_{f\in I(E)} \ker((\id\otimes f)\Delta_\G).$$
    All of the work for this part of the claim has already been done. For $x\in j(E)^\perp$ and $f\in I(E)$, notice that we can repeat the steps of the converse above to get $f*x = 0$. Conversely, notice the forward implication above actually depended on having $f*x = 0$.
    
    2. For $\pi\in Irr(\G)$, pick an ONB $\{e_j^\pi\}$ by choosing an ONB for $E_\pi$ and then extending it to $\fH_\pi$. Then for $f\in I(E)$ we have
    $$0 = \pi(f)(E_\pi) = [f(u_{i,j}^{E_\pi})](E_\pi)$$
    if and only if $f(u_{i,j}^\pi) = 0$ for every $e_j^\pi\in E_\pi$. The rest is clear.
\end{proof}
Using the duality observed between right invariant subspaces of $L^\infty(\G)$ and left ideals in $L^1(\G)$ in Section $2.3$, Proposition \ref{idealsinL^1} and the above proposition gives us an equivalent statement in $L^\infty(\G)$.
\begin{cor}
    Let $\G$ be a CQG. If $X\trianglelefteq_l L^\infty(\G)$ is a right invariant subspace, then there exists a hull $E$ such that
    $$L^\infty(\G,E)\subseteq X\subseteq \bigcap_{j\in I(E)}\ker((\mathrm{id}\otimes f)\Delta_\G).$$
\end{cor}
\begin{proof}
    This follows immediately from Proposition \ref{Annj(E)} and Proposition \ref{idealsinL^1}.
\end{proof}
Before getting to the main theorem, we still need to think about the singly generated ideals.
\begin{lem}\label{SinglyGenIdeals}
    Let $\G$ be a CQG. Fix $f\in L^1(\G)$ and let $E$ be the hull of $\G$ associated with the closed principal left ideal $\overline{L^1(\G)*f}$. The following hold:
    \begin{enumerate}
        \item we have $E_\pi = \ker(\pi(f))$ for each $E_\pi\in E$;
        \item $f\in I(E)$;
        \item if $E$ is a set of synthesis, then $f\in \overline{L^1(\G)*f}$;
        \item $\overline{L^1(\G)*f}^\perp = \ker((\id\otimes f)\Delta_\G)$;
    \end{enumerate}
\end{lem}
\begin{proof}
    1. This follows easily from the fact $\pi(g*f)(E_\pi) = \pi(g)\pi(f)(E_\pi) = 0$ for each $g\in L^1(\G)$.
    
    2. This follows immediately by definition of $I(E)$ and from 1.
    
    3. If $E$ is a set of synthesis, then from 1., 2., and Proposition \ref{idealsinL^1},
    $$\overline{j(E)} = \overline{L^1(\G)*f} = I(E)\ni f.$$
    4. If $f*x = 0$, then $g*f(x) = g(f*x) = 0$ for each $f\in L^1(\G)$, that is, $x\in (\overline{L^1(\G)*f})^\perp$. Conversely, if $x\in (\overline{L^1(\G)*f})^\perp$, then $0 = g*f(x) = g(f*x)$ for all $g\in L^1(\G)$, which implies $f*x = 0$.
\end{proof}
\begin{thm}\label{StrcClosedLeftIdeals}
    Let $\G$ be a CQG. Then every hull is a set of synthesis if and only if $f\in \overline{L^1(\G)*f}$ for all $f\in L^1(\G)$.
\end{thm}
\begin{proof}
    If we assume every hull is a set of synthesis, then in particular, from Lemma \ref{SinglyGenIdeals} we have $f\in \overline{L^1(\G)*f}$ for every $f\in L^1(\G)$. Conversely, because of Proposition \ref{Annj(E)}, all we need to show is $I(E)\subseteq [\bigcap_{f\in I(E)}\ker((\id\otimes f)\Delta_\G)]_\perp$. So, take $f\in I(E)$ and let $x\in L^\infty(\G)$ satisfy $f*x = 0$. Find a net $(g_i)\subseteq L^1(\G)$ such that $g_i*f\to f$. Then
    $$0 = g_i(f*x) = g_i*f(x)\to f(x).$$
\end{proof}
From Proposition \ref{idealsinL^1}, the hull of a closed left ideal $I$ in $L^1(\G)$ is a set of synthesis if and only if $I$ is the unique closed left ideal corresponding to $E$. So, when every $E$ is a set of synthesis, we have a description of every closed left ideal of $L^1(\G)$ in terms of the hull $E$. To be explicit, from Theorem \ref{StrcClosedLeftIdeals} we immediately conclude the following.
\begin{cor}\label{leftIdealsCor}
    Let $\G$ be a CQG such that $f\in \overline{L^1(\G)*f}$ for all $f\in L^1(\G)$. The closed left ideals of $L^1(\G)$ are of the form $I(E)$ for some hull $E$.
\end{cor}
In light of Theorem \ref{StrcClosedLeftIdeals}, we make the following definition.
\begin{defn}
    We say a LCQG $\G$ has {\bf Ditkin's left property at infinity (or property left $D_\infty$)}, if $f\in \overline{L^1(\widehat{\G})*f}$ for every $f\in L^1(\widehat{\G})$.
\end{defn}
\begin{rem}
    The reason property left $D_\infty$ is apparently a property of $\widehat{\G}$ instead of $\G$ is so that the definition generalizes the definition of property $D_\infty$ in the case where $\G = G$ is a locally compact group (see \cite{A20}). In that case, property $D_\infty$ is regarded as a property of the locally compact group $G$ instead of its dual $\widehat{G}$. Note the similarity of this convention with the definitions of weak amenability and the AP (Section $2.4$).
\end{rem}
\subsection{Weak$^*$ Closed Left Ideals of Measure Algebras}
The main result of this subsection is that we achieve a characterization of the weak$^*$ closed ideals of the measure algebra of a coamenable CQG. Essentially, we will show White's techniques \cite{Wh18} generalize to the setting of CQGs. Before getting to this, however, we will begin by discussing some more general things about ideals of measure algebras of (not necessarily coamenable) CQGs.

First note, given a hull $E$ of a CQG $\G$, we have the closed left ideal $I(E)$ of $L^1(\G)$. Then, using that $L^1(\G)\subseteq M^r(\G), M^u(\G)$ isometrically as an ideal, we have the weak$^*$ closed left ideals $\overline{I(E)}^{wk*}\subseteq M^r(\G)$ and $\overline{I(E)}^{wk*}\subseteq M^u(\G)$. As we will see shortly, for coamenable $\G$, this construction gives us all weak$^*$ closed left ideals of $M(\G)$.

Because we have the embedding $\Pol(\G)\subseteq C^u(\G)$, we can immediately extend $\pi\in Irr(\G)$ to a representation $\pi : M^u(\G)\to M_{n_\pi}$ by setting
$$\pi(\mu) = (\mu\otimes\id)(U^\pi) = [\mu(u_{i,j}^\pi)].$$
With this in hand, given a hull $E$, we will define
$$I^u(E) = \{\mu\in M^u(\G) : \pi(\mu)(E_\pi) = 0, ~ E_\pi\in E\}$$
and
$$I^r(E) = \{\mu\in M^r(\G) : \pi(\mu)(E_\pi) = 0, ~ E_\pi\in E\},$$
which are both easily checked to be weak$^*$ left closed ideals in $M^u(\G)$ and $M^r(\G)$ respectively. We also have the following.
\begin{prop}
    Let $\G$ be a CQG. Then $I^u(E)= C^u(\G,E)^\perp$ and $I^r(E) = C^r(\G,E)^\perp$.
\end{prop}
\begin{proof}
    The proof follows similarly to the analogous result in Proposition \ref{Annj(E)}.
\end{proof}
Now we will work towards the main result. The techniques involve exploiting the following sort of objects.
\begin{defn}
    A Banach algebra $A$ is {\bf compliant} if there exists a Banach space $X$ such that $M(A) = X^*$ and the maps $M(A)\to A$, $\mu\mapsto \mu a$ and $\mu\mapsto a\mu$, for $a\in A$, are weak*-weakly continuous (where weak$^* = \sigma(M(A), X)$).
\end{defn}
Recall, for coamenable $\G$ we have $M(\G) = M^l(L^1(\G))$ (cf. \cite{HNR11}).
\begin{prop}
    Let $\G$ be a coamenable LCQG. Then $L^1(\G)$ is compliant if and only if $\G$ is compact.
\end{prop}
\begin{proof}
    According to \cite[Proposition 5.8 (i)]{Wh18}, compliance of $L^1(\G)$ implies it is an ideal in $L^1(\G)^{**}$. Then \cite[Theorem 3.8]{R09} implies $\G$ is compact. Conversely, if $\G$ is compact then, thanks to \cite[Theorem 2.3]{R08},
    $$C(\G) = L^1(\G)*L^\infty(\G) = L^\infty(\G)*L^1(\G)$$
    where we have used Cohen's factorization theorem. From this, \cite[Lemma 5.7]{Wh18} says $L^1(\G)$ is compliant.
\end{proof}
As we will see in the proof of the following theorem, as observed in \cite[Theorem 5.10]{Wh18}, the distinct advantage of compliance of $L^1(\G)$ is that the weak$^*$ closed left ideals in $M(\G)$ are exactly of the form $\overline{I}^{wk*}$ where $I$ is a left ideal in $L^1(\G)$. This allows us to use the structure of closed ideals in $L^1(\G)$ obtained in the previous section.
\begin{thm}\label{wk*IdealsMeasures}
    Let $\G$ be a coamenable CQG. The weak* closed left ideals of $M(\G)$ are of the form $I^u(E)$ for some hull $E$ of $\G$.
\end{thm}
\begin{proof}
    Since $L^1(\G)$ is compliant, because of \cite[Theorem 5.10]{Wh18}, the weak* closed left ideals of $M(\G)$ are of the form $\overline{I}^{wk*}$ for a closed left ideal $I$ of $L^1(\G)$. Now apply Corollary \ref{leftIdealsCor} to get that $I = I(E)$ for some hull $E$. By definition $L^1(\G)\cap I^u(E) = I(E)$, and so using \cite[Theorem 5.10]{Wh18}, we get $I^u(E) = \overline{I}^{wk*}$ as desired.
\end{proof}
\begin{cor}
    Let $\G$ be a coamenable CQG. The closed right invariant subspaces of $C(\G)$ are of the form $C(\G,E)$ for some hull $E$ of $\G$.
\end{cor}
\begin{proof}
    This follows from Theorem \ref{wk*IdealsMeasures} and Proposition \ref{Annj(E)}.
\end{proof}
\subsection{Ditkin's Property at Infinity and Examples}
Even for locally compact groups, property left $D_\infty$ is a rather opaque condition and, to our knowledge, there are no known examples of locally compact groups with property left $D_\infty$ \cite[Section 6.7]{KL19}. Recently a characterization of property left $D_\infty$ for locally compact co--groups has been obtained by Andreou \cite{A20}. Using the techniques developed there, Andreou obtained a new proof that AP implies property left $D_\infty$ using techniques based around Fubini tensor products (a result which may also be read from \cite[Theorem 1.11]{KH94}). For this section, we will write down some basic equivalent formulations of property left $D_\infty$ (which were recorded by Andreou for locally compact groups). Then we will provide examples of CQGs with property left $D_\infty$.

We will say $x\in L^\infty(\G)$ satisfies {\bf condition $(H)$} if $x\in \overline{L^1(\G)*x}^{wk*}\subseteq L^\infty(\G)$.
\begin{prop}
    Let $\G$ be a CQG. TFAE:
    \begin{enumerate}
        \item $\widehat{\G}$ has property left $D_\infty$;
        \item every $x\in L^\infty(\G)$ satisfies condition $(H)$;
        \item for $x\in L^\infty(\G)$ and $f\in L^1(\G)$, $f*x = 0$ implies $f(x) = 0$;
        \item and for all $X\trianglelefteq_r L^\infty(\G)$ we have $x*f\in X$ for all $f\in L^1(\G)$ if and only if $x\in X$.
    \end{enumerate}
\end{prop}
\begin{proof}
    First, $(4\implies 2)$ follows verbatim to the corresponding statement in \cite[Proposition 6.7]{A20}. Now we note that commutativity of the Fourier Algebra appears in the proof of the corresponding statement in \cite[Proposition 6.7]{A20} of $(1\implies 3)$, so we must supply our own proof in the CQG setting here to obtain $(1\iff 2\iff 3)$. With that said, $(2\implies 1)$ does follow verbatim from \cite[Proposition 6.7]{A20} and the converse follows from a similar Hahn--Banach argument. Then $(1\iff 3)$ follows from the observation
    $$f\in \overline{L^1(\G)*f}\iff \ker(\id\otimes f)\Delta_\G = \overline{L^1(\G)*f}^\perp \subseteq\{f\}^\perp.$$
    So, we have $(4\implies 3\iff 2\iff 1)$.\\
    \\
    For $(3\implies 4)$, take $f\in L^1(\G)$, so $x*f\in X$, which means for $g\in X_\perp$ that
    $$0 = g(x*f) = f(g*x).$$
    Since $f\in L^1(\G)$ was arbitrary, we deduce that $g*x = 0$, which means $g(x) = 0$, that is $x\in X$ as desired.
\end{proof}
The following is an immediate consequence of Proposition \ref{APai}.
\begin{prop}
    If a DQG $\widehat{\G}$ has the AP, then it has property left $D_\infty$.
\end{prop}
Now we point out weakly amenable examples in the literature.
\begin{eg}
    The duals of the free unitary and orthogonal compact quantum groups $U^+_F$ and $O^+_F$, quantum permutation groups $S^+_n$, and quantum reflection groups $H^{+(s)}_n$ are all weakly amenable \cite{B16}.
\end{eg}
\section{Coamenability and Ideals}
\subsection{Compact Quasi--Subgroups}
In this section, we make progress towards understanding the closed left ideals of $L^1(\G)$ that admit {\bf bounded right approximate identities (brais)}. Recall that if $\Gamma$ is discrete, then the ideals $I(\Lambda)\subseteq A(\Gamma)$, where $\Lambda$ is a subgroup of $\Gamma$, admit a {\bf bounded approximate identity (bai)} if and only if $\Gamma$ is amenable \cite{For}. Moreover, it is not too difficult to prove that $I(s\Lambda)$ has a bai if and only if $I(\Lambda)$ does, where $s\in \Gamma$. So, if $\Gamma$ is amenable, we have identified many ideals that admit bais and otherwise, many ideals that do not. Note that this was generalized to amenable locally compact groups in \cite{FKLS03}.

Recall from Section $2.3$ that the right coideals of $VN(\Gamma)$ are of the form $VN(\Lambda)$, where $\Lambda$ is a subgroup of $\Gamma$ and $J^1(VN(\Lambda)) := VN(\Lambda)_\perp = I(\Lambda)$. So, for CQGs in general, we replace $VN(\Lambda)$ with a right coideal $N\subseteq L^\infty(\G)$, and $I(\Lambda)$ with $J^1(N) := N_\perp$. As we will elaborate on shortly, we are forced to restrict our attention to a certain subclass of right coideals. This starts with the following result of \cite{NM19}.
\begin{thm}\cite[Theorem 3.1]{NM19}\label{QuasiSubJust}
    Let $\G$ be a LCQG and $I\trianglelefteq L^1(\G)$ be a closed left ideal. Then $I$ has a brai only if there exists a right $L^1(\G)$--module projection $L^\infty(\G)\to I^\perp$.
\end{thm}
\begin{proof}
    See the corresponding reference for a proof. We will point out, however, that the projection onto $I^\perp$ is of the form $x\mapsto e*x$, where $e\in L^\infty(\G)$ is a weak$^*$ cluster point of the given brai, and $e*x$ denotes the natural action of $L^\infty(\G)^*$ on $L^\infty(\G)$.
\end{proof}
\begin{rem}\label{princIdealsMr}
    According to \cite[Theorem 3.1]{HNR12}, (recall, $\G$ is compact) the bounded right $L^1(\G)$--module maps $\fB_{L^1(\G)}^R(L^\infty(\G))$
    are normal, i.e., for every $M\in \fB_{L^1(\G)}^R(L^\infty(\G))$ there exists a left $L^1(\G)$-module map $m : L^1(\G)\to L^1(\G)$ such that $M = m^* : L^\infty(\G)\to L^\infty(\G)$.
    Now, if $I$ is a closed left that has a brai with weak$^*$ cluster point $e\in L^\infty(\G)^*$ (afforded by Banach--Alaoglu), then we get $I = L^1(\G)*e$, where $f*e$ denotes the natural action of $f$ on $L^\infty(\G)^*$ (see also the proof of \cite[Theorem 2.2]{NM19}).
\end{rem}
We will be focusing on those left ideals admitting brais for which their corresponding right $L^1(\G)$-module projections are positive. Thus we are in a special case of the necessary condition in Theorem \ref{QuasiSubJust}.
\begin{defn}
    A {\bf compact quasi--subgroup} of $\G$ is a right coideal $N$ of $L^\infty(\G)$ such that there exists a normal right $L^1(\G)$--module conditional expectation $E_N^R : L^\infty(\G)\to N$.
\end{defn}
A {\bf closed quantum subgroup} of a CQG $\G$ is a CQG $\H$ such that there exists surjective unital $*$-homomorphism $\pi_\H^u : C^u(\G)\to C^u(\H)$ satisfying
$$(\pi^u_\H\otimes\pi^u_\H)\Delta^u_\G = \Delta_\H^u\circ \pi_\H^u.$$
Equivalently we can find such a $*$-homomorphism $\pi_\H : \Pol(\G)\to \Pol(\H)$. The quotient space $\G/\H$ is defined by setting
$$\Pol(\G/\H) = \{a\in \Pol(\G/\H) : (\id\otimes \pi_\H)\Delta_\G(a) = a\otimes 1\}.$$
We denote $L^\infty(\G/\H) = \overline{\Pol(\G/\H)}^{wk*}$ etc.

The quotients associated with closed quantum subgroups of a CQG fall under the umbrella of compact quasi-subgroups. Using the embedding $(\pi_\H^u)^* : M^u(\H)\to M^u(\G)$ we find that $\omega_{\G/\H} = h_\H^u\circ \pi_\H^u\in M^u(\G)$ is an idempotent state. Then the adjoint of the map $f\mapsto f*\omega_{\G/\H}$ is a normal right $L^1(\G)$-module conditional expectation $L^\infty(\G)\to L^\infty(\G/\H)$.
\begin{rem}
\begin{enumerate}
    \item It turns out that the compact quasi--subgroups are in $1$--$1$ correspondence with the idempotent states in $M^u(\G)$ (cf. \cite{Soltan}). The relationship is as follows: the compact quasi--subgroups take the form $M^l_{\omega}(L^\infty(\G))$ where $\omega\in M^u(\G)$ is an idempotent state and we recall that $M^l_\omega$ is the adjoint of the left multiplier $m^l_\omega\in M^l(L^1(\G))$ associated with $\omega$ (cf. Section 2.6).
    
    Note in particular that if $\H$ is a closed quantum subgroup of $\G$, then $L^\infty(\G/\H) = M^l_{h_\H^u\circ\pi^u_\H}(L^\infty(\G))$, so $L^\infty(\G/\H)$ is a compact quasi-subgroup of $\G$. We also point out that
    $$M^l_{\omega_N}(L^\infty(\G))_\perp = \overline{\{f- f*\omega : f\in L^1(\G)\}}^{||\cdot||_1}$$
    where $\omega_N$ is the idempotent state associated with $N$ (cf. \cite{NSSS19}).
    \item It is a theorem of Kawada and It\^{o} \cite{KI40} that the closed subgroups of a compact group $G$ are in $1$--$1$ correspondence with the idempotent states in the measure algebra via $H\leq G \iff \omega_H\in M(G)$, where $L^\infty(G/H) = M^l_{\omega_H}(L^\infty(G))$. Likewise, for a discrete group $\Gamma$, we have a $1$--$1$ correspondence between subgroups and idempotent states in $B(\Gamma)$ via $\Lambda\leq \Gamma\iff 1_\Lambda\in B(\Gamma)$, and we have $VN(\Gamma)\cdot 1_\Lambda = VN(\Lambda)$ (see \cite{IS05}). So, we can view the notion of a compact quasi--subgroup as the quantum analogue of the quotient algebra of a closed subgroup of a compact group / group von Neumann algebra of a subgroup of a discrete group. Note that, in general, not all compact quasi--subgroups arise from closed quantum subgroups. We can see this even for discrete co--groups as any non--normal subgroup (of the underlying discrete group) gives rise to a compact quasi--subgroup but not a closed quantum subgroup. A non-cocommutative example can be found in \cite{P96}.
    
    \item Not every right coideal of a CQG is a compact quasi-subgroup. The Podle\`s spheres of $SU_q(2)$, where $q\neq 2$, give uncountably many examples for a fixed $q$ (see \cite{P87}).
\end{enumerate}
\end{rem}
Recall, for an invariant subalgebra $N$, $N_*$ has a Banach algebra structure inherited directly from the quotient $L^1(\G)/J^1(N) \cong N_*$. If we assume $N_*$ has a bai (so in the case $N = L^\infty(\G/\H)$ for some normal closed quantum subgroup $\H$, $\G/\H$ is coamenable), then we can easily transfer bais between $L^1(\G)$ and $J^1(N)$ from the results found in \cite{DW79} and \cite{For87}.
\begin{prop}\label{G/HCoamen}
    Let $I\trianglelefteq L^1(\G)$ be a closed two--sided ideal and suppose $L^1(\G)/I$ has a bai. Then $\G$ is coamenable if and only if $I$ has a bai.
\end{prop}
\begin{proof}
    If $L^1(\G)/I$ and $I$ both have bais, then we can build a bai for $L^1(\G)$ \cite[Pg. 43]{DW79}. The converse is covered by the more general fact that given a Banach algebra $A$ which has a bai and closed left ideal $J$, $J$ has a brai if and only if $J^\perp$ is right invariantly complemented in $A^*$, i.e., there is a right $A$--module projection $P : A^*\to J^\perp$ (cf. \cite[4.1.4 Pg. 42]{For87}).
\end{proof}
When $\Gamma$ is a discrete group and $\Lambda$ is a subgroup, the compact quasi-subgroup $VN(\Lambda)\subseteq VN(\Gamma)$ is explicitly written in terms of its hull $\Lambda$ via
$$VN(\Lambda) \cong \overline{\mathrm{Span}\{\lambda_\Gamma(s) : s\in \Lambda\}}^{wk*} = L^\infty(\widehat{\Gamma}, \Lambda),$$
where $L^\infty(\widehat{\Gamma},\Lambda)$ is our notation from Section $3.1$. Our techniques for the main results of this section will use a similar description for the compact quasi-subgroups of CQGs in general.

Given a compact quasi--subgroup $N$, we will denote the corresponding idempotent state by $\omega_N$. Then, for compact $\G$, we have a projection
$$\tilde{E}^R_N := (E^R_N)|_{\Pol(\G)} = (\id\otimes\omega_N)\Delta_\G : \Pol(\G)\to \Pol(N) := \tilde{E}^R_N(\Pol(\G))$$
onto a right invariant subalgebra of $\Pol(\G)$ satisfying $$\overline{\Pol(N)}^{wk*} = N.$$
See also \cite[Section 2]{FLS16} for a discussion in the case of CQGs.
\begin{rem}\label{NormalSGMin}
    Wang \cite{W09} showed that normality is equivalent to having $[\omega_{L^\infty(\G/\H)}(u^\pi_{i,j})] = I_{n_\pi}$ or $0$ for all $\pi\in Irr(\G)$, from which it was also shown for normal $\H$,
    $$\Pol(\G/\H) = \Pol(\G, E_\H)$$
    where $E_\H = (E_\pi)_{\pi\in Irr(\G)}$ is the hull such that $E_\pi = \fH_\pi$ if $[\omega_{L^\infty(\G/\H)}(u^\pi_{i,j})] = I_{n_\pi}$ and $E_\pi = \{0\}$ otherwise. In particular, $L^\infty(\G, E_\H) = L^\infty(\G/\H)$.
\end{rem}
The above remark generalizes to the following for compact quasi--subgroups (and uses the same techniques as Wang).
\begin{lem}\label{RepsQSubs}
    Let $N$ be a compact quasi--subgroup. Then there exists an orthonormal basis $\{e_i^\pi\}$ of $\fH_\pi$ so that $u^\pi_{i,j}\in N$ if and only if $\omega_N(u^\pi_{j,j}) = 1$, and\normalfont{
    $$\Pol(N) = \Span\{u_{i,j}^\pi : 1\leq i\leq n_\pi, e_j^\pi\in E_\pi\}.$$}
\end{lem}
\begin{proof}
    Fix $\pi\in Irr(\G)$. Since $\omega_N$ is an idempotent state, $\pi(\omega_N)$ is an orthogonal projection. Choose an ONB $\{e_i^\pi\}$ so that $\pi(\omega_N) = [\omega_N(u_{i,j}^\pi)]$ is diagonal. Then, we have $\omega_N(u_{i,j}^\pi) = \delta_{i,j}$ or $\omega_N(u_{i,j}^\pi) = 0$ for any $i,j$. So, if $u^\pi_{i,j}\in N$, then
    $$u^\pi_{i,j} = E^r_N(u^\pi_{i,j}) = (\id\otimes\omega_N)\Delta_\G(u^\pi_{i,j}) = \sum_{t=1}^{n_\pi} u_{i,t}^\pi \omega_N(u_{t,j}^\pi) = \omega_N(u^\pi_{j,j})u^\pi_{i,j}$$
    implies $\omega_N(u^\pi_{j,j}) = 1$. If $u^\pi_{i,j}\notin N$, then
    $$u^\pi_{i,j}\neq E_N^r(u_{i,j}^\pi) = \omega_N(u^\pi_{j,j})u^\pi_{i,j},$$
    which means $\omega_N(u^\pi_{j,j}) = 0$.
    
    Notice that we have shown $E^R_N(u^\pi_{i,j}) = u^\pi_{i,j}$ or $0$. The second claim follows. 
\end{proof}
\begin{cor}\label{RepsQSubsDense}
    Let $\G$ be a CQG and $N$ a compact quasi--subgroup with hull $E_N$. Then \normalfont{$\Pol(\G, E_N) = \Pol(N)$},
    and furthermore, $L^\infty(\G, E_N) = N$.
\end{cor}
This establishes an explicit description of compact quasi-subgroups in terms of their underlying hull that is reminiscent of the embedding $VN(\Lambda)\subseteq VN(\Gamma)$ when $\Lambda$ is a subgroup of $\Gamma$.

Now fix a compact quasi--subgroup $N$. We will build from it canonical ``continuous function spaces'' and ``measure spaces''. Through these spaces we identify a certain weak$^*$ closed left ideal, $J^u(N)$, in $M^u(\G)$ corresponding to $N$. The importance of this left ideal will reveal itself in the main results of this section. Inspired by the techniques of White \cite{Wh18} that we exploited in Section $3.2$, we relate the problem of determining the existence of brais in $J^1(N)$ with the problem of determining when $\overline{J^1(N)}^{wk*} = J^u(N)$.

Accordingly, we will define
$$C^u(N) := \overline{\Pol(N)}^{||\cdot||_u} ~ \text{and} ~ C^r(N) := \Gamma_\G(C^u(N)),$$
where we recall that $\Gamma_\G: C^u(\G)\to C^r(\G)$ is the reducing morphism, and so, by definition, we have a surjective $*$--homomorphism $\Gamma_\G|_{C^u(N)} : C^u(N)\to C^r(N)$. Note that since $\Gamma_\G(\Pol(N))\subseteq N$, we have $C^r(N)\subseteq N$ and by weak density of $\Pol(N)$ in $N$, we have $\overline{C^r(N)}^{wk*} = N$. We also have the right $M^u(\G)$--module conditional expectation
$$E^R_{C^u(N)} := (\id\otimes\omega_N)\Delta^u_\G : C^u(\G)\to C^u(N).$$
Then we will set
$$M^u(N) := C^u(N)^* ~ \text{and} ~ M^r(N)^* := C^r(N).$$
Then, by definition, the adjoint is a completely isometric embedding:
$$(\Gamma_\G|_{C^u(N)})^* : M^r(N)\to M^u(N).$$
Now, by taking the adjoint of the inclusion $C^u(N)\subseteq C^u(\G)$, we obtain a surjective weak$^*$--weak$^*$ continuous linear map
$$T^u_N : M^u(\G)\to M^u(N)$$
whose kernel we denote by $J^u(N)$, which of course satisfies $J^u(N) = C^u(N)^\perp$.
\begin{rem}
    Note that in the case of a quotient $\G/\H$, it is not hard to show that
    $$C^u(\G/\H) := \{a\in C^u(\G) : (\id\otimes \pi^u_\H)\Delta^u_\G(a) = a\otimes 1\} = C^u(L^\infty(\G/\H)).$$
    Furthermore, we will denote $C^r(\G/\H) = C^r(L^\infty(\G/\H))$ etc.
\end{rem}
For the moment we will consider quotients of closed quantum subgroups. The following notion was formulated in \cite{Kal2}.
\begin{defn}
    For a CQG $\G$, we say a quotient $\G/\H$ is {\bf coamenable} if $\pi_\H^u : C^u(\G)\to C^u(\H)$ admits a reduced version, that is, there exists $\pi^r_\H : C^r(\G)\to C^r(\H)$ such that $\Gamma_\H\circ\pi^u_\H = \pi^r_\H\circ\Gamma_\G$ (where $\Gamma_\G :C^u(\G)\to C^r(\G)$ is the reducing morphism).
\end{defn}
When $\Gamma$ is discrete, it is well-known that a subgroup $\Lambda$ is amenable if and only if $1_\Lambda\in B_r(\Gamma)$. So, in general, coamenability of $\G/\H$ is a bonafide generalization of amenability for ``quantum quotients.''

Now we state a useful necessary condition for coamenability of a quotient motivating the condition $\omega_N\in M^r(\G) = \overline{L^1(\G)}^{wk*}$, which we will be using for the main results of this subsection.
\begin{prop}\label{CoamenQuots}
    Let $\G$ be a CQG and $\H$ a closed quantum subgroup. If $\G/\H$ is coamenable, then $\omega_{L^\infty(\G/\H)}\in M^r(\G)\subseteq M^u(\G)$.
\end{prop}
\begin{proof}
    Recall that $\omega_{L^\infty(\G/\H)} = h_\H^u\circ\pi^u_\H$ and the completely isometric embedding $M^r(\G)\subseteq M^u(\G)$ is given by the adjoint of $\Gamma_\G$. Recall also that we can factorize $h^u_\H = h^r_\H\circ\Gamma_\H$. Then
    $$M^r(\G)\ni h^r_\H\circ\pi^r_\H\circ\Gamma_\G = h_\H^r\circ\Gamma_\H\circ\pi^u_\H = h^u_\H\circ\pi^u.$$
\end{proof}
So, by assuming $\omega_N\in M^r(\G)$, we know that this condition holds at least for coamenable quotients (compare Corollary \ref{CoamenBAIIdeal} with Proposition \ref{G/HCoamen}). Next we take a look at the associated left ideals in $M^u(\G)$.
\begin{prop}\label{RightUnitUniv}
    Let $\G$ be a CQG and $N$ a compact quasi--subgroup. Then $J^u(N)$ has a right unit.
\end{prop}
\begin{proof}
    Let $\G$ be a CQG and $N$ a compact quasi--subgroup. First notice that for $\mu\in J^u(N)$ and $a\in C^u(\G)$,
    $$0 = \mu(E^R_{C^u(N)}(a)) = \mu(\id\otimes\omega_N)\Delta^u_\G(a) = \mu*\omega_N(a)$$
    Therefore, $\mu*(\epsilon^u_\G - \omega_N) = \mu$ for all $\mu\in J^u(N)$. Finally, by choosing an ONB as in Lemma \ref{RepsQSubs},
    $$(\epsilon_\G^u - \omega_N)(u^\pi_{i,j}) = \delta_{i,j} - \delta_{i,j} = 0$$
    for all $u^\pi_{i,j}\in \Pol(N)$. Then, from density of $\Pol(N)$ in $C^u(N)$, we have $\epsilon^u_\G - \omega_N|_{C^u(N)} = 0$, that is, $\epsilon^u_\G - \omega_N\in J^u(N)$.
\end{proof}
\begin{cor}
    Let $\G$ be a CQG and $N$ an invariant subalgebra. Then $J^u(N)$ has an identity element.
\end{cor}
\begin{proof}
    A similarly proof to Proposition \ref{RightUnitUniv} shows $\epsilon_\G^u - \omega_N$ is also a left identity.\\
\end{proof}
Accordingly, we will denote the right (or two--sided when appropriate) identity of $J^u(N)$ by $e^u$. Notice then that
$$J^u(N) = J^u(N)*e^u\subseteq M^u(\G)*e^u\subseteq J^u(N),$$
so
$$J^u(N) = M^u(\G)*e^u.$$
A natural question to ask is, when can we approximate $e^u$ from $J^1(N)$? In particular, how does this relate to the existence of a brai in $J^1(N)$? The answer is as follows.
\begin{thm}\label{Wk*IdealsBrais}
    Let $N$ be a compact quasi--subgroup. If $J^1(N)$ has a brai then $\overline{J^1(N)}^{wk*} = J^u(N)$, where the weak* topology is the one induced by $C^u(\G)$.
\end{thm}
\begin{proof}
    Assume $J^1(N)$ has a brai $(e_j)$ and pass to a weak$^*$ convergent subnet with limit point $e\in\overline{J^1(N)}^{wk*}\subseteq M^r(\G)$. Before proceeding with the proof, we point out some intermediate facts. We will first show $\overline{J^1(N)}^{wk*} = M^r(\G)*e$. Since $L^1(\G)$ is an ideal in $M^r(\G)$, for $\mu\in M^r(\G)$ and $f\in J^1(N)$ we have $\mu*f*e_j\in J^1(N)$ for all $j\in J^1(N)$, from which we conclude that $\mu*f\in J^1(N)$. In particular, we have $\mu*e_j\in J^1(N)$ for all $j$ and so by taking limits, $\mu*e\in \overline{J^1(N)}^{wk*}$.
    
    Next we will show
    \begin{align}
        J^u(N)*e = \overline{J^1(N)}^{wk*}.\label{Wk*IdealEq}
    \end{align}
    Clearly $J^1(N)\subseteq J^u(N)$, from which we immediately deduce $e*e^u = e$. Then,
    $$J^1(N)\subseteq J^u(N)*e = J^u(N)*e*e\subseteq M^r(\G)*e = \overline{J^1(N)}^{wk*},$$
    using that $M^r(\G)$ is an ideal in $M^u(\G)$, and so $J^u(N)*e = \overline{J^1(N)}^{wk*}$ as desired.
    
    Set $\omega^r_N = \epsilon^u_\G - e$. For $f\in L^1(\G)$, we have
    $$f\circ M^l_{\omega^r_N} = 0 \iff f*\omega_N^r = 0 \iff f*e = f.$$
    So, $f\circ M^l_{\omega_N^r} = 0$ for all $f\in J^1(N)$, which implies $M^l_{\omega_N^r}(L^\infty(\G))\subseteq N$. Then, since $M^l_{\omega_N}|_N = \id_N$,
    $$M^l_{\omega_N^r*\omega_N} = M^l_{\omega_N}\circ M^l_{\omega_N^r} = M^l_{\omega_N^r}.$$
    Recall from the proof of Proposition \ref{RightUnitUniv} that $\omega_N = \epsilon^u_\G - e^u$. Then, by injectivity of $\mu\mapsto M^l_{\mu}$ we get
    $$(\epsilon^u_\G - e)*(\epsilon_\G^u - e^u) = \omega_N*\omega_N^r = \omega_N^r = \epsilon_\G^u - e,$$
    from which we have $e*e^u = e^u$. Therefore, using \eqref{Wk*IdealEq},
    $$\overline{J^1(N)}^{wk*} = \overline{J^1(N)}^{wk*}*e^u = J^u(N)*e*e^u = J^u(N)*e^u = J^u(N).$$
\end{proof}
Our question of weakly approximating elements of $J^u(N)$ by elements of $J^1(N)$ turns out to relate to coamenability of $\G$.
\begin{thm}\label{CoamenUniv}
    Let $\G$ be a CQG and $N$ a compact quasi--subgroup. If $\G$ is coamenable then $\overline{J^1(N)}^{wk*} = J^u(N)$. Conversely, if $\omega_N\in M^r(\G)$ and $\overline{J^1(N)}^{wk*} = J^u(N)$, then $\G$ is coamenable.
\end{thm}
\begin{proof}
    Assume $\G$ is coamenable. We first note that $C^u(N) = C^r(N)$, so we will simply write $C(N)$. Because of \cite[Theorem 5.10]{Wh18} (cf. Theorem \ref{wk*IdealsMeasures}), it suffices to show $J^u(N)\cap L^1(\G) = J^1(N)$. First, clearly $J^1(N)\subseteq J^u(N)$. For the reverse containment, take $a\in C(\G)$. Then for $a\in C(N)$ and $f\in J^u(N)\cap L^1(\G)$,
    \begin{align*}
        0 = T^u_N(f)(a) &= f(a) = T_N(f)(a)
    \end{align*}
    which implies $f(N) = 0$ by weak$^*$ density of $C(N)$ in $N$ and normality of $f$.
    
    Conversely, 
    $$M^r(\G) = \overline{L^1(\G)}^{wk*}\ni \omega_N + e^u  = \epsilon^u_\G,$$
    where the equality was noted in the proof of Proposition \ref{RightUnitUniv}. This implies coamenability of $\G$.
\end{proof}
A coamenability result we are looking for presents itself as follows.
\begin{cor}\label{CoamenBAIIdeal}
    Let $\G$ be a CQG and $N$ an compact quasi--subgroup such that $\omega_N\in M^r(\G)$. Then $J^1(N)$ has a brai if and only if $\G$ is coamenable.
\end{cor}
\begin{proof}
    If $J^1(N)$ has a brai, then apply Theorems \ref{Wk*IdealsBrais} and \ref{CoamenUniv} to get coamenability of $\G$. The converse is a special case of the following more general fact: if $A$ is a Banach algebra with a bai, then a closed left ideal $J$ has a brai if and only if there is a right $A$--module projection $A^*\to J^\perp$ (cf. \cite[4.1.4 Pg. 42]{For87}).
\end{proof}
From Proposition \ref{CoamenQuots} and Corollary \ref{CoamenBAIIdeal}, we also deduce the following.
\begin{cor}
    Let $\G$ be a CQG and $\H$ a closed quantum subgroup such that $\G/\H$ is a coamenable quotient. Then $\G$ is coamenable if and only if $J^1(\G,\H)$ has a brai.
\end{cor}
A compact quasi--subgroup $N$ is {\bf open} if $\omega_N\in L^1(\G)$. It was shown in \cite{Soltan} that the open quasi--subgroups of a CQG are the finite dimensional right coideals. Using Corollary \ref{CoamenBAIIdeal} we obtain the following.
\begin{cor}
    Let $\G$ be a CQG and $N$ an open compact quasi--subgroup. Then $\G$ is coamenable if and only if $J^1(N)$ has a brai.
\end{cor}

\subsection{Quantum Cosets of Compact Quasi--Subgroups}
Our main result of the previous subsection is a generalization of the result of Forrest \cite{For} that for a discrete group $\Gamma$, the ideal $I(\Lambda)$ has a bai if and only if $\Gamma$ is amenable, restricted to the case where $\Lambda$ is an {\it amenable} subgroup of $G$ (so that $1_\Lambda\in B_r(\Gamma)$). The techniques of Forrest exploited the fact that $I(\Lambda)$ has a bai if and only if $I(s\Lambda)$ has a bai and the use of the Hahn-Banach theorem, for $s\in \Gamma\setminus \Lambda$.

As a way to find a version of Corollary \ref{CoamenBAIIdeal} without the condition $\omega_N\in M^r(\G)$, we will consider (invariant) compact quasi-subgroups that ``admit a quantum coset''. Along the way, we achieve a generalization of the fact $I(\Lambda)$ has a bai if and only if $I(s\Lambda)$ has a bai, and so we are able to characterize coamenability of $\G$ in terms of the closed left ideals coming from the ``quantum cosets'' of compact quasi-subgroups. A quantum coset will turn out to be a translation of a compact quasi-subgroup by an element of the intrinsic group of $\G$.
\begin{defn}
    The group
    $$Gr(\G) = \{x\in L^\infty(\G)^{-1} : \Delta_\G(x)= x\otimes x\},$$
    where $L^\infty(\G)^{-1}$ is the set of invertibles in $L^\infty(\G)$, is called the {\bf intrinsic group} of $\widehat{\G}$.
\end{defn}
\begin{rem}
    Our reference for the following discussion is \cite{Kal3}. We actually have that each element of $Gr(\G)$ is unitary and $Gr(\G)$ is a locally compact group when equipped with the weak$^*$ topology. It is straightforward seeing that $Gr(\G) = sp(L^1(\G))$. Alternatively, one can identify $Gr(\G)\subseteq Irr(\G)$ as the $1$--dimensional unitary representations, so in particular, we have $Gr(\G)\subseteq \Pol(\G)$. Note that whenever $\G$ is compact, the von Neumann algebra generated by $Gr(\G)$ is of the form $VN(\Gamma)$ for some discrete group $\Gamma$. We will abuse notation and simply write $Gr(\G) = \Gamma$.
\end{rem}
When $\Gamma$ is discrete and $\Lambda$ is a subgroup, we have
$$\lambda(s)VN(\Lambda) = \overline{\Span\{\lambda(st) : t\in \Lambda\}}^{wk*} = L^\infty(\widehat{\Gamma}, s\Lambda).$$
Then $I(s\Lambda) = (\lambda(s)VN(\Lambda))_\perp$. So, more generally, for a CQG $\G$ and compact quasi-subgroup $N\subseteq L^\infty(\G)$, we consider $sN$, where $s\in Gr(\G)$, to be a {\bf quantum coset} of $N$.

We will proceed by writing down a series of lemmas that allow us to use a generalization of Forrest's argument in \cite{For}, as alluded to at the start of the subsection. To begin, the following essentially says that the quantum cosets of a compact quasi-subgroup are disjoint from the compact quasi-subgroup. This is reminiscent of the fact that proper cosets of a subgroup are disjoint with the subgroup.
\begin{lem}\label{DisjointCosets}
    Let $N$ be a compact quasi--subgroup. Then for $x\in Gr(\G)$,
    $$xN\cap N = \begin{cases} N &\text{if} ~ x\in N\\ \{0\} &\text{otherwise}\end{cases}.$$
\end{lem}
\begin{proof}
    From Lemma \ref{RepsQSubs} we know that $\omega_N(x) = 1$ if $x\in N$ and $\omega_N(x) = 0$ if $x\notin N$. Then the equation
    $$(\id\otimes\omega_N)\Delta_\G(x) = \omega_N(x)x$$
    tells us $E^R_N(x) = x$ if $x\in N$ and $E^R_N(x) = 0$ otherwise. Then for $y\in N$, using that $E^R_N$ is a conditional expectation, we have
    $$xy = E^R_N(xy) = E^R_N(x)y$$
    if and only if $E^R_N(x) = x$.
\end{proof}
Given $x\in L^\infty(\G)$, we denote $f\cdot x\in L^1(\G)$ as the action such that $(f\cdot x)(y) = f(xy)$ for all $y\in L^\infty(\G)$. If $x\in Gr(\G)$, then, since $x$ is a unitary, $\cdot x : L^1(\G)\to L^1(\G)$ is an isometric algebra automorphism.
\begin{lem}\label{Ideals}
    Let $X\trianglelefteq_r L^\infty(\G)$ be a right invariant weak$^*$ closed subspace. For $x\in Gr(\G)$,
    $$X_\perp\cdot x = (xX)_\perp$$
    is a closed left ideal. If $X_\perp$ is two--sided, then $X_\perp\cdot x$ is two--sided.
\end{lem}
\begin{proof}
    Since $(X_\perp)^\perp = X$, it is clear that $X_\perp\cdot x\subseteq (xX)_\perp$. For $f\in (xX)_\perp$, it can be shown using a Hahn--Banach argument that $f\cdot x^{-1}\in X_\perp$. Then $f = (f\cdot x^{-1})\cdot x\in X_\perp\cdot x$. For the remaining claim, it is easy to see that $X_\perp\cdot x$ is closed. Then for $f\in L^1(\G)$ and $y\in X$,
    $$(yx)*f = (f\otimes \id)(x\otimes x)\Delta_\G(y)\in xX$$
    because $((f\cdot x)\otimes \id)\Delta_\G(y)\in X$. So $xX$ is right invariant, meaning $(xX)_\perp = X_\perp\cdot x$ is a left ideal. If $X$ is also left invariant, then left invariance of $xX$ follows similarly.\\
\end{proof}
\begin{rem}\label{cosetsQSubs}
    Let $N$ be a compact quasi--subgroup. Note that for $x\in Gr(\G)\setminus N $, $xN$ does not contain $1$ and so cannot be a von Neumann algebra and is not a compact quasi--subgroup. We did see, however, in the above lemma that $xN$ is a weak$^*$ closed right invariant subspace of $L^\infty(\G)$. Next we will note $M^l_{\omega\cdot x^{-1}} : L^\infty(\G)\to L^\infty(\G)$ is a projection onto $xN$. To see this, first notice $x\Pol(N)$ is weak$^*$ dense in $xN$ because $L^\infty(\G) \ni y\mapsto xy\in L^\infty(\G)$ is a weak$^*$--weak$^*$ homeomorphic linear bijection. Therefore it suffices to check $(\id\otimes\omega_N\cdot x^{-1})\Delta_\G$ is a projection onto $x\Pol(N)$. For this, if we take $y\in \Pol(N)$,
        $$(\id\otimes\omega_N\cdot x^{-1})\Delta_\G(xy) = x(\id\otimes\omega_N)\Delta_\G(y) = xy$$
        and if we take $y\in\Pol(\G)$,
        $$(\id\otimes\omega_N\cdot x^{-1})\Delta_\G(y) = x(\id\otimes\omega_N)\Delta_\G(x^{-1}y)\in x\Pol(N).$$
        We point out that the idempotent functional $\omega_N\cdot x^{-1}$ is easily seen to be a contractive idempotent. Contractive idempotents and their associated weak$^*$ closed right invariant subspaces were studied in \cite{NSSS13, K18} (at the level LCQGs). While given a contractive idempotent $\omega\in M^u(\G)$, $M^l_\omega(L^\infty(\G))$ is not an algebra, it is a {\bf ternary ring of operators (TRO)}, i.e., whenever $x,y,z\in M^l_\omega(L^\infty(\G))$, $xy^*z\in M^l_\omega(L^\infty(\G))$.
\end{rem}
\begin{lem}\label{SurjectIdeals}
    Let $N$ be a compact quasi--subgroup. For $x\in Gr(\G)\cap (L^\infty(\G)\setminus N)$, $T_N(J^1(N)\cdot x) = N_*$.
\end{lem}
\begin{proof}
    For each $y\in N$, using $xN\cap N = \{0\}$ from Lemma \ref{DisjointCosets}, find $f\in J^1(N)\cdot x$ so that $f(y)\neq 0$. Then $T_N(f)(y) = f(y)\neq 0$, from which, using a straightforward Hahn--Banach argument and that $T_N$ is open (open mapping theorem) and hence closed, we see that $T_N(J^1(N)\cdot x) = N_*$ as desired.
\end{proof}
The following theorem is the statement that $\G$ is coamenable if and only if the preannihilator of an invariant quantum coset has a bai.
\begin{thm}\label{CoamenIdeal}
    Let $\G$ be a CQG and $X$ a weak$^*$ closed invariant subspace of $L^\infty(\G)$. Suppose $\{sX : s\in Gr(\G)\}$ has a compact quasi--subgroup and at least two elements. Then $\G$ is coamenable if and only if $X_\perp$ has a bai.
\end{thm}
\begin{proof}
    Let $N\in \{sX : s\in Gr(\G)\}$ denote the compact quasi--subgroup. As discussed in Remark \ref{cosetsQSubs}, we know $N = x_0X$ for some $x_0\in Gr(\G)$ and from Lemma \ref{Ideals}, $J^1(N) = X_\perp\cdot x_0$ is a two--sided ideal (and so $N$ is actually an invariant compact quasi--subgroup).
    
    The proof is a generalization of the argument employed by Forrest \cite{For}. Suppose $X_\perp$ has a bai. Now, for $f\in X_\perp$, $y\in L^\infty(\G)$, and $x\in Gr(\G)$
    \begin{align*}
        ||(e_j\cdot x)*(f\cdot x) - f\cdot x||_1 &= \sup_{y\in B_1(L^\infty(\G))}| (e_j\otimes f)\Delta_\G(xy) - f(xy)|
        \\
        &= \sup_{y\in x^{-1}B_1(L^\infty(\G))} |(e_j\otimes f)\Delta_\G(y) - f(y)|
        \\
        &= \sup_{y\in B_1(L^\infty(\G))}|e_j*f(y) - f(y)|
        \\
        &= ||e_j*f - f||_1\to 0
    \end{align*}
    where in the second last equality, we used the fact $x$ is a unitary. A similar proof shows $f*(e_j\cdot x)\to f$, so $e_j\cdot x$ is a bai on $X_\perp\cdot x$. Now, we know $X\cdot Gr(\G)$ has two elements, one of which is $N$. Without loss of generalization, we will suppose the other element is $X$. So, we have that $J^1(N)$ and $X_\perp = J^1(N)\cdot x_0^{-1}$ both have bais. Then from invariance of $N$, we know $T_N$ is an algebraic homomorphism and coupling this fact with Lemma \ref{SurjectIdeals} finds us a bai on $N_* = T_N(J^1(N)\cdot x_0^{-1})$. Then we apply Proposition \ref{G/HCoamen}.
    
    For the converse, from the discussion in Remark \ref{cosetsQSubs}, we have a right $L^1(\G)$--module projection $L^\infty(\G)\to X$ induced by the idempotent functional $\omega_N\cdot x_0^{-1}$. The rest is identical to the proof of the converse of Proposition \ref{G/HCoamen}.
\end{proof}
With our last corollary, we drop the invariance condition of $N$, but we are forced to put back the condition that $\omega_N\in M^r(\G)$. What is distinct from before is that we incorporate quantum cosets of $N$.
\begin{cor}\label{CoamenIdealCor}
    Let $\G$ be a CQG and $X$ a weak$^*$ closed right invariant subspace of $L^\infty(\G)$. Suppose $\{sX : s\in Gr(\G)\}$ has a compact quasi-subgroup $N$ such that $\omega_N\in M^r(\G)$ and at least two elements. Then $\G$ is coamenable if and only if $X_\perp$ has a brai.
\end{cor}
\begin{proof}
    Let $N\in \{sX : s\in Gr(\G)\}$ be the given compact quasi--subgroup. In the proof of Theorem \ref{CoamenIdeal}, we showed $J^1(N)$ also has a brai. Then from Corollary \ref{CoamenBAIIdeal} we know $\G$ is coamenable. The proof of the converse is identical to the proof of the converse in Theorem \ref{CoamenIdeal}.
\end{proof}
\begin{rem}
    For the general definition of an open quantum subgroup, we direct the reader to \cite{Kal1} (and \cite{Daws} for the general notion of a closed quantum subgroup), which includes the notion of a quantum subgroup of a discrete quantum group. Let $\G$ be a CQG. Whenever $\widehat{\H}$ is a closed quantum subgroup of $\widehat{\G}$, $L^\infty(\H)\subseteq L^\infty(\G)$ is an invariant coideal and every invariant coideal is of the form $L^\infty(\H)$ where $\widehat{\H}$ is a quantum subgroup of $\widehat{\G}$. In particular, it follows that the hypothesis of Theorem \ref{CoamenIdeal} is satisfied whenever $\widehat{\H}$ is a quantum subgroup of $\widehat{\G}$ such that $Gr(\G)\setminus Gr(\H)$ is non-trivial.
\end{rem}
\subsection{Examples: Discrete Crossed Products}
The main results of Section $4.2$ are applicable only to CQGs where the intrinsic group is non-trivial, and avoids the compact quasi-subgroups in question. Unfortunately, it can be the case that the intrinsic group is trivial, which happens exactly when the trivial representation is the only $1$-dimensional representation (eg. $SU_q(2)$). In this subsection we present a class of examples that do have a non-trivial intrinsic group, and are examples for which the results of Section $4.2$ apply. These examples come from CQGs that are crossed products of CQGs by discrete groups.

With what follows, we use \cite[Chapter X]{Tak03} and \cite{Dav} as a reference.
\begin{defn}
    A {\bf discrete $W^*$-dynamical system} is a triple $(M, \Gamma,\alpha)$ where $M$ is a von Neumann algbera, $\Gamma$ is a discrete group, and $\alpha : \Gamma\to Aut(M)$ is a weak$^*$ continuous homomorphism.
\end{defn}
Given a discrete $W^*$--dynamical system, we denote the finitely supported $M$--valued functions on $\Gamma$
$$M[\Gamma] = \Span\{ as : a\in M, s\in \Gamma\}.$$
We view the symbols $a\in M$ and $s\in \Gamma$ as the elements $a = ae$ and $1s = s$ in $M[\Gamma]$, which we assert to satisfy $sas^{-1} = \alpha(s)(a)$ for all $s\in \Gamma$ and $a\in M$, and has the following $*$--algebraic structure
$$(as)(bt) = a\alpha(s)(b)s^{-1}t ~ \text{and} ~ (as)^* = s^{-1}a^*$$
for $t\in \Gamma$ and $b\in M$ (note that we only needed $M$ to be a unital $*$--algebra). In other words, $M[\Gamma]$ is a $*$--algebra that contains a copy of $M$ and a copy of $\Gamma$ as unitaries such that $\alpha$ is inner.

A {\bf covariant representation} of $(M,\Gamma,\alpha)$ is a pair $(\pi_M,\pi_\Gamma)$ such that $\pi_M : M\to \fB(\fH)$ is a normal $*$-representation and $\pi : \Gamma\to \fB(\fH)$ a unitary representation satisfying the covariance equation
$$\pi_M(\alpha(s)(a)) = \pi_\Gamma(s)\pi_M(a)\pi_\Gamma(s)^*.$$
A covariant representation gives rise to a representation $\pi_M\rtimes_\alpha\pi_\Gamma : M[\Gamma]\to \fB(\fH)$ by setting $\pi_M\rtimes_\alpha\pi_\Gamma(as) = \pi_M(a)\pi(s)$. If we let $\theta : M\to \fB(\fH_\theta)$ be a faithful $*$-representation, then we can define a canonical covariant representation $(\pi^\theta,\lambda^\theta)$ of $(M,\Gamma,\alpha)$ by defining
$$\pi^\theta : M\to \fB(L^2(\Gamma,\fH_\theta)), ~ \pi^\theta(a)\xi(s) = \theta(\alpha(s^{-1})(a))\xi(s)$$
and
$$\lambda^\theta : \Gamma\to \fB(L^2(\Gamma,\fH_M)), ~ \lambda^\theta(t)\xi(s) = \xi(t^{-1}s)$$
for $a\in M$, $s,t\in \Gamma$, and $\xi\in L^2(\Gamma,\fH_\theta)$. Then
$$M\overline{\rtimes}_\alpha \Gamma := (\pi^\theta(M)\lambda^\theta(\Gamma))''$$
is the {\bf discrete von Neumann crossed product} of $(M, \Gamma,\alpha)$. We note that $M\overline{\rtimes}_\alpha \Gamma$ contains an isometric copy of $M$ and $VN(\Gamma)$ and we will abuse notation and denote the copy of each $x\in M$ and $s\in \Gamma$ by $x$ and $s$ respectively.

If $\G$ is a CQG and $\alpha$ satisfies $(\alpha\otimes\alpha)\Delta_\G  = \Delta_\G\circ\alpha$, $L^\infty(\G)\overline{\rtimes}_\alpha\Gamma$ is the von Neumann algebra of a CQG. We will call $(L^\infty(\G),\Gamma,\alpha)$ a {\bf Woronowicz $W^*$-dynamical system}.
\begin{thm}\label{CrossedProds}\cite{Wang1, Wang2, BV05}
    Let $(L^\infty(\G),\Gamma,\alpha)$ be a Woronowicz $W^*$-dynamical system. Then there exists a CQG (denoted $\G\rtimes \widehat{\Gamma}$) such that:
    \begin{enumerate}
        \item $Irr(\G\rtimes_\alpha\widehat{\Gamma}) =\{su_{i,j}^\pi : \pi\in Irr(\G), ~ s\in \Gamma\}$;
        \item $\Pol(\G)[\Gamma] = \Pol(\G\rtimes_\alpha \widehat{\Gamma})$;
        \item $L^\infty(\G)\overline{\rtimes}_\alpha \Gamma = L^\infty(\G\rtimes_\alpha\widehat{\Gamma})$;
        \item $h_{\G\rtimes_\alpha\widehat{\Gamma}} = h_\G\rtimes_\alpha 1_{\{e\}}$;
        \item $\Delta_{\G\rtimes_\alpha\widehat{\Gamma}}|_{L^\infty(\G)} = \Delta_\G$ and $\Delta_{\G\rtimes_\alpha\widehat{\Gamma}}|_{VN(\Gamma)} = \Delta_{\widehat{\Gamma}}$.
    \end{enumerate}
\end{thm}
\begin{rem}
\begin{enumerate}
    \item We see from $5.$ that every right invariant subspace of $L^\infty(\G)$ and $VN(\Gamma)$ is also a right invariant subspace of $L^\infty(\G\rtimes_\alpha\widehat{\Gamma})$.
    \item We obtain many examples from the combination of any CQG $\G$ and discrete group $\Gamma$. Consider the trivial action $\id : \Gamma\to Aut(L^\infty(\G))$, which is defined by $\id(s)(x) = x$ for all $s\in\Gamma$ and $x\in L^\infty(\G)$. In this case, we get $L^\infty(\G\rtimes_{id}\widehat{\Gamma}) \cong L^\infty(\G)\overline{\otimes}VN(\Gamma)$ as von Neumann algebras.
\end{enumerate}
\end{rem}\label{DCPs}
    Our main application of the results from Section $4.2$ lies in the following.
    \begin{prop}\label{DisCrosProdEgs}
        Let $(L^\infty(\G),\Gamma,\alpha)$ be a Woronowicz $W^*$-dynamical system. Then the following hold:
        \begin{enumerate}
            \item if we assume $Gr(\G) \neq \{1\}$, $\G\rtimes\Gamma$ is coamenable if and only if $J^1(VN(\Gamma))\cdot x$ has a bai, where $x\in Gr(\G)$;
            \item $\G\rtimes\Gamma$ is coamenable if and only if any $J^1(N)\cdot s$ has a bai, where $s\in \Gamma$ and $Ns = VN(\Lambda s)$ for some proper subgroup $\Lambda$ of $\Gamma$ or $Ns$ is an (invariant) quantum coset of $L^\infty(\G)$.
        \end{enumerate}
    \end{prop}
    \begin{proof}
        First note that $\Gamma, Gr(\G)\subseteq Gr(\G\rtimes\Gamma)$ and $Gr(\G)\cap \Gamma = \{1\}$. For $1$, because $Gr(\G)\neq \{1\}$, we can find $x\in Gr(\G)\subseteq Gr(\G\rtimes\Gamma)\setminus \Gamma$ and we apply Theorem \ref{CoamenIdeal}.\\
        \\
        Likewise, for $2$, we can find $x\in\Gamma\subseteq Gr(\G\rtimes\Gamma)\setminus Gr(\G)$ or non--trivial $x\in \Lambda\setminus\Gamma\subseteq Gr(\G\rtimes\Gamma)$, and then we apply Theorem \ref{CoamenIdeal}.
    \end{proof}
\section{Open Problems}
We will present problems left over from our investigations.

We have characterized the CQGs where every hull is a set of synthesis (Theorem \ref{StrcClosedLeftIdeals}) as the CQGs with property left $D_\infty$. This means the closed left ideals (and consequently the weak$^*$ closed right invariant subspaces of $L^\infty(\G)$) are classified for the CQGs satisfying property left $D_\infty$. This leaves us with the following very open ended question.
\begin{ques}
    Which hulls of a CQG are always sets of synthesis?
\end{ques}
For example, the closed subgroups of a locally compact groups are always sets of synthesis (cf. \cite{KL19}). So, we ask the following more specific question.
\begin{ques}
    Are the hulls of right coideals sets of synthesis?
\end{ques}
We have made partial progress towards identifying when the left ideals $J^1(N)$ associated with a compact quasi--subgroup admit a brai. While we have a complete characterization in terms of the condition $J^u(N) = \overline{J^1(N)}^{wk*}$ (Theorem \ref{Wk*IdealsBrais}), our characterization in terms of coamenability of $\G$ (Corollary \ref{CoamenBAIIdeal}) requires what is essentially a coamenability type condition on $N$. This leaves us with the following question.
\begin{ques}
    Given a CQG and compact quasi--subgroup $N$, if $J^1(N)$ has a brai, then do we have $\omega_N\in M^r(\G)$?
\end{ques}
Successfully answering the above question means we can say $\G$ is coamenable if and only if $J^1(N)$ admits a brai.

We have also characterized coamenability of $\G$ in terms the existence of brais on the associated left ideals of a very small class of TROs associated with a contractive idempotent (Theorem \ref{CoamenIdeal} and Corollary \ref{CoamenIdealCor}). Namely, if we set $X = M^l_\omega(L^\infty(\G))$ where $\omega\in M^u(\G)$ is a contractive idempotent, we require $Gr(\G)\cap (L^\infty(\G)\setminus X)\neq \varnothing$ and one of two things: either $X$ is invariant or $\omega\in M^r(\G)$. Therefore we ask the following general question.
\begin{ques}
    Let $\G$ be a CQG and $\omega\in M^u(\G)$ a contractive idempotent. Do we have that $M^l_\omega(L^\infty(\G))_\perp$ admits a brai if and only if $\G$ is coamenable?
\end{ques}

\section*{Acknowledgments}
I thank Nico Spronk, Brian Forrest, and Michael Brannan for their supervision and support over the fruition of this project. I would like to thank Nico especially for his careful reading of this article and helpful comments. I would also like to thank Jared White for initiating the Groups, Operators, and Banach Algebras (GOBA) webinar in the Winter of 2020. The ideas that built this work were seeded from Jared's talk at the onset of the GOBA webinar. Finally, I thank the anonymous referee for their feedback which greatly improved the presentation of this paper. This work was supported by a QEII-GSST scholarship. 

{\small\bibliography{IdealsofCQGs}}
\bibliographystyle{amsplain}
\renewcommand*{\bibname}{References}

\end{document}